\documentclass{article}

\usepackage{zibtitlepage}

\usepackage[utf8]{inputenc}
\usepackage{amsmath, amsfonts, amssymb, amsthm}
\usepackage{mathtools}
\usepackage{calc}
\usepackage{graphicx}
\usepackage{dsfont}
\usepackage[acronym]{glossaries}

\usepackage[table]{xcolor}
\usepackage{enumerate}

\usepackage{tikz}
\usetikzlibrary{circuits}
\usetikzlibrary{quotes, angles, positioning}
\usepackage{pgfgas}

\usepackage[a4paper, margin={25mm}]{geometry}

\usepackage{accents}
\usepackage{url}

\usepackage{natbib}
\setcitestyle{authoryear,square,semicolon}

\newlength\myheight
\newlength\mydepth
\newcommand*\inlinegraphics[1]{%
  \settototalheight\myheight{Xygp}%
  \settodepth\mydepth{Xygp}%
  \raisebox{-\mydepth}{\includegraphics[height=\myheight]{#1}}%
}

\newcommand{\R}{\mathds{R}}

\newacronym{PDE}{PDE}{system of partial differential equations}
\newacronym{ODE}{ODE}{system of ordinary differential equations}
\newacronym{DAE}{DAE}{system of differential-algebraic equations}
\newacronym{ROOT}{ROOT}{system of nonlinear equations}

\newcommand{\RN}[1]{%
  \textup{\uppercase\expandafter{\romannumeral #1}}%
}

\theoremstyle{plain}
\newtheorem{theorem}{Theorem}[section]
\newtheorem{assumption}[theorem]{Assumption}
\newtheorem{observation}[theorem]{Observation}

\newcommand{\SIunit}[1]{\mathrm{#1}}
\newcommand{\SInumber}[2]{{#1}\,\SIunit{#2}}

\input{commands}

\newcommand{\skipOver}[1]{}

\newcommand{\CV}{\mathcal{RG}}

\begin{document}

\ZTPAuthor{
\ZTPHasOrcid{Felix Hennings}{0000-0001-6742-1983}, 
\ZTPHasOrcid{Milena Petkovic}{0000-0003-1632-4846}, 
\ZTPHasOrcid{Tom Streubel}{0000-0003-4917-7977}}
\ZTPTitle{On the Numerical Treatment of Interlaced Target Values -- Modeling, Optimization and Simulation of Regulating Valves in Gas Networks}
\ZTPNumber{21-32}
\ZTPMonth{December}
\ZTPYear{2021}
\ZTPInfo{}

\title{On the Numerical Treatment of Interlaced Target Values -- Modeling, Optimization and Simulation of Regulating Valves in Gas Networks}
\author{
\ZTPHasOrcid{Felix Hennings}{0000-0001-6742-1983}\and\
\ZTPHasOrcid{Milena Petkovic}{0000-0003-1632-4846}\and\
\ZTPHasOrcid{Tom Streubel}{0000-0003-4917-7977}}
\hypersetup{pdftitle={On the Numerical Treatment of Interlaced Target Values -- Modeling, Optimization and Simulation of Regulating Valves in Gas Networks},
  pdfauthor={Felix Hennings, Milena Petkovic, Tom Streubel}}

\zibtitlepage
\maketitle

\begin{abstract}
Due to the current and foreseeable shifts in energy production, the trading and transport operations of gas will become more dynamic, volatile, and hence also less predictable.
Therefore, computer-aided support in terms of rapid simulation and control optimization will further broaden its importance for gas network dispatching.
In this paper, we aim to contribute and openly publish two new mathematical models for regulators, also referred to as control valves, which together with compressors make up the most complex and involved types of active elements in gas network infrastructures.
They provide full direct control over gas networks but are in turn controlled via target values, also known as set-point values, themselves.
Our models incorporate up to six dynamical target values to define desired transient states for the elements' local vicinity within the network.
That is, each pair of every two target values defines a bounding box for the inlet pressure, outlet pressure as well as the passing mass flow of gas.
In the proposed models, those target values are prioritized differently and are constantly in competition with each other, which can only be resolved dynamically at run-time of either a simulation or optimization process.
Besides careful derivation, we compare simulation and optimization results with predictions of the commercial simulation tool SIMONE.
\end{abstract}

\section{Introduction}
The physical and technical properties of gas transport in networks have been studied for many decades already. 
They have gotten increasing attention in the last years since countries worldwide attempt a turnaround in their national energy policies towards carbon dioxide neutral means of energy production.
Reasons for this are the potential usage of gas networks as flexible energy storage for balancing the highly volatile generation of renewable energy sources \citep{Fed2021} or the prospect of shifting to hydrogen transport in the future \citep{EurLex2020}.
In both cases, controlling the network is expected to become more dynamic and complex, either due to the increasing variability of supply and demand or the gas properties of hydrogen compared to natural gas~\citep{HopHenZitGot2020}.
To make the task of operating gas networks manageable, the decision-making is split into multiples levels.
It ranges from the complete view on the network featuring only abstract representations of the technical elements to detailed but local control decisions, for example, the degree of openness of a single pressure regulator or the rotational speed of a gas compressor~\citep{Ste1988}.
The decisions on the different levels are usually connected via the concept of \emph{set-point values} or \emph{target values}.
These represent a control requirement determined in the full network view, which then serves as an objective for the local operation of the actual machinery.

Controlling an element with set-points is a well-known concept in the community of automated control systems. The difference between the measured (actual) and the desired (target) value of an observed process variable is the basis for error-controlled regulation with negative feedback (closed-loop) in automatic control~\citep{Porter1988}. Proportional-Integral-Derivative (PID) control is the most common feedback control algorithm in engineering systems. In process industries, most of the control loops ($\approx95\%$) are of PID type~\citep{Astrom2008}.
The advanced control methods including model based control techniques (model predictive control~\citep{Camacho2007,DARBY2012328,HABER20145333} or multivariable control~\citep{Skogestad2005}), intelligent control techniques (fuzzy control~\citep{Nguyen2019} or neural network based control~\citep{Meireles2003}), and adaptive control~\citep{Astrom2014}, are often used as a high level control procedures for designing high performance controllers that can be applied to high-order and multivariable processes, typically nonlinear and subject to constraints.

When looking at work regarding the simulation or optimization of gas networks, the usage of set-points to control active elements has a long history as well.
In~\citep{MalFinBulRol1993}, the authors present the optimization problem of finding a stationary control for the British gas network.
The control of the compressors is defined as the corresponding set-point for the gas pressure after compression, also called the outlet or outgoing pressure of the element.
The same holds for \citep{RacCar2000}, which examine a transient optimization problem on a gunbarrel network, given a desired final state of the network.
In a second paper, the same authors extended the model to prepare for not one but multiple possible future scenarios in a robust way \citep{CarRac2003}.
In recent years, more set-point or target value models have been used.
In their description of a stationary gas network control optimization problem, \citep{PfeFueGeiGei2014,Schmidt2015,KocHilPfeSch2015} introduced a model for a regulator without remote access having a fixed set-point for the outgoing pressure.
However, depending on whether the actual state value exceeds the set-point, matches it, or falls below it, the regulator is closed, active, or fully open.
In the last case, it acts as a bypass.
Finally, there are also models including a multitude of different set-point or target values.
In \citep{PamBolDij2016}, a model used for gas network simulation was introduced, in which regulators and compressors are given a set-point either for the inlet pressure, the outlet pressure, the flow, or the ratio of both the in- and outlet pressure.
Alternatively, they can be chosen to be closed or in bypass mode.
The simulation model of \citep{Benner2019} features two different modes for each compressor and regulator.
These either try to keep the inlet or the outlet pressure at a given target value.

In this article, we introduce a much more involved description for the target value control of active elements.
It features up to 6 different set-point values, which are in simultaneous use and hence competing.
As a result, the element can react to changing conditions in the surrounding network without the need to switch the target value objective manually.
For example, it is able to autonomously switch between holding the inlet pressure, holding the flow, or even closing the element, in the case of a regulator.
This model captures the behavior of complex elements used in real-world networks and industry-standard simulators to precisely express the desired control given dynamically changing conditions in the corresponding part of the network.

The description can, in general, be applied to regulators and compressors.
However, we will focus on regulators here, i.e., the elements reducing the pressure in flow direction by reducing the element's throughput.
In literature, there exist several models for regulators. 
For example, \citep{Huck2018Perturbation} models an idealized regulator by demanding that the pressure difference between its outlet and inlet pressure must be positive and inside a given interval.
In \citep{Benner2019}, the regulator is described as a resistor with variable diameter.
Hence, it is given a degree of openness ranging between 0 and 1.
This allows them to include the edge case of a fully closed regulator, in which the two pressures are decoupled and the flow is zero.
The model presented in  \citep{PfeFueGeiGei2014,Schmidt2015,KocHilPfeSch2015} distinguishes two types of regulators: Those with and without remote access.
The ones without remote access are controlled as already described above using the outlet pressure target value.
When having remote access, the regulators can either be closed or again keep their pressure difference inside a given interval.
Another and precursor to one of our models presented here is presented in \citep{HenAndHopTur2021}.
Within there, the authors consider three different modes for a regulator: The regulator is either closed, active, or open.
A closed regulator is defined as for the other models, an open one does assume the inlet and outlet pressures to be equal, while an active regulator allows for arbitrary pressure reduction.
In addition, the regulator features a check valve or flap trap, which prevents flow against the regulator's orientation.

We will present two models for regulators using the target value control: One model suitable for dynamic simulation using a given set of target values over time and one for discrete optimization problems, in which target values should be determined that induce the desired element behavior while changing as little as possible.

The rest of the paper is structured as follows:
In Section~\ref{sec:pipeModel}, we will first introduce the fundamental gas flow dynamic in pipes and our representation of a gas network as a directed graph.
Based on this, we describe our new regulator model based on the 6 different target values in Section~\ref{sec:target_value_model}.
In the following two sections, we first introduce a corresponding model suitable for the transient simulation of regulators given the target value settings for a future time horizon. Afterwards, we describe a discrete optimization model that can determine the control of a regulator in terms of the target values in a transient context.
In Section~\ref{sec:numeric_evaluation}, we present two experiments on a minimal gas network and compare the performance of the two models against each other as well as a commercial simulation software.
We conclude with the outlook in Section~\ref{sec:outlook}.

\section{Models for Gas Flow in Pipe Networks}
\label{sec:pipeModel}
Mathematical models of various so-called active elements, such as regulators, compressors, or valves, manipulate the gas flow through themselves by blocking, resisting, or enhancing it.
However, the response to their actions is determined in the surrounding pipes that provide the volume for the gas to be.
Hence, we need to model the behavior of the gas in pipes in terms of pressures \(p\), mass-flows \(q\), and their interaction.
We will use a version of the isothermal Euler equations, which we will introduce and describe in the following subsection.

\subsection{Isothermal Euler Equations}
A simplified or \emph{friction dominated} version of the one-dimensional set of isothermal Euler equations for the description of the behavior of fluids within perfectly cylindrical pipes of length \(L \in \R_{>0}\) in \(\SIunit{m}\) and diameter \(D \in \R_{>0}\) in \(\SIunit{m}\) is given by
\begin{subequations}\label{eqn:pipe_pde}
\begin{align}
    \dot p &= -\left[\frac{z^2(p)}{z(p) - p\cdot \partial_p z(p)}\right]\kappa\cdot \partial_x q, \label{eqn:continuity} \\ 
    \chi\cdot \dot q &= A\partial_x p + \frac{\lambda}{2 D}\kappa z(p) \frac{q|q|}{p} + \delta\!h \frac{g}{\kappa} \frac{p}{z(p)}, \label{eqn:momentum} 
\end{align}
\end{subequations}
with \(\chi \in\{0, 1\}\) being a time constant model switch.
These particular versions are referred to as ISO2 if \(\chi = 1\), or as ISO3 if \(\chi = 0\) according to \citep{DomHilLanMeh2021} and can be derived by a time scaling approach presented by \citep{Brouwer2011}.
According to empirical observations in \citep{Hen2021}, we will use \(\chi = 0\) exclusively here.

The intrinsic quantities of the equations are \(p \equiv p(x, t)\) for the pressure, \(q \equiv q(x, t)\) for the mass flow, \(\rho(x, t)\) for the density, and \(v(x, t)\) for the velocity and are parameterized along the longitudinal axis by \(x \in \Omega = [0, L]\) as well as being averaged across the cross-sectional area \(A = \frac 14 \pi D^2\) as enclosed by the pipe. 
Furthermore and throughout, let \( \dot p(x, t) \equiv \partial_t p(x, t) \) and likewise for \(\dot q(x, t)\) be the time derivatives. 
Also, let \(R_s\) be the specific gas constant, \(T\) be the time-constant temperature, \(\kappa = R_sT/A\) as shorthand, \(g \approx \SInumber{9.81}{m/s^2}\) the gravity constant, and \(\delta\! h = (h_r - h_\ell)/L\) as secant slope defined by the heights \(h_\ell\) on the left and \(h_r\) right end of the pipe. Furthermore, let \(z(p) = z_0 + z_1 p + \dots + z_np^n\) be a polynomial model or truncated Virial expansion for the gas factor for real gases \citep{Virial_expansion}, and \(\lambda\) be the friction coefficient of the \emph{Darcy-Weisbach equation} \citep{Bro2003}. The latter is an empirical model of pressure loss due to friction with the pipe wall. 

For the determination of coefficients of the polynomial real gas factor model \(z(p)\), we will make use of the linear AGA formula \citep{SIMONE, DomHilLanMeh2021} of the American Gas Association, being accurate up to pressures \(p \le \SInumber{70}{bar}\), the quadratic formula of Papay \citep{Pap1968, Sal2002} for pressures \(p \le \SInumber{150}{bar}\), or the constant model \(z_0 = c^2/(R_s T)\) as suggested by \citep{DomHilLanMeh2021}, where \(c\) is the speed of sound within the given gas mixture.

To approximate the friction factor \(\lambda\), we will use the \emph{Nikuradse} formula \citep{Nik1950}, which is given by
\begin{align}
    \lambda = \left[ -2\log_{10}\big( r/(3.71D) \big) \right]^{-2}.
\end{align}
It is an explicit simplification of the otherwise implicit \emph{Colebrook-White} \citep{ColWhi1937, Bro2003} formula. Here \(r\) is the roughness of the pipe wall surface and is usually provided in millimeters \(\SIunit{mm}\).

Two important identities that have already been utilized for the ultimate derivation of \eqref{eqn:pipe_pde} are 
\begin{align}
    p &= R_s\cdot T\cdot \rho\cdot z(p) & \text{and}& & q &= A\cdot \rho\cdot v, \label{eqn:stateEqn}
\end{align}
where the first one is referred to as the \emph{state equation for real gases} \citep{Men2005}.
We call \eqref{eqn:pipe_pde} the \emph{pipe equations}.
More in-depth details on Euler equations in the context of fluid transport can be found, e.g., in \citep{Osi1996, DomHilLanMeh2021, Benner2019}. 

Alternatively, in \citep{Hen2018}, a linearization of the pipe equation \eqref{eqn:pipe_pde} has been proposed, which has the benefit of avoiding divisions by small pressure values. 
There it was observed that the equations \eqref{eqn:stateEqn} yield the identity \(|v| = z(p)\cdot \kappa\cdot |q|/p\), which can be fixed on predetermined or forecast velocities \(|\bar v|\), such that
\begin{align}
    0 &= A\partial_x p + \frac{\lambda}{2 D} \cdot q|\bar v| + \delta\!h \frac{g}{\kappa} \frac{p}{z(p)}. \tag{\ref{eqn:pipe_pde}b''} \label{eqn:new_momentum_const_v}
\end{align}
Provided we use a constant model for the real gas factor, then equations \eqref{eqn:continuity} and \eqref{eqn:new_momentum_const_v} together yield a linearized model of the pipe equations.

\subsection{Gas Networks as Graph Structures From a Macroscopic Perspective}
We have just established the system of pipe equations \eqref{eqn:pipe_pde}, but we are interested in gas networks as a whole and in solving online tasks. 
In this context, online means to utilize simulation or optimization to formulate control and operation recommendations as a guiding as well as continuous process live at operation time.
Thus, the overall problem results in potentially large systems that must be solved repeatedly in a very limited time frame.
For this reason, we have to take a macroscopic viewpoint regarding modeling the network and its building blocks. Hence, we consider gas networks as directed graphs \(\mathcal G = (\mathcal N, \mathcal A)\) over some time horizon \([t_0, T] \subseteq \R\) of interest.
The set of nodes \(\mathcal N\) is finite \(|\mathcal N| < \infty\), and the set of arcs \(\mathcal A \subseteq \mathcal N\times \mathcal N\) is a collection of ordered tuples of nodes.

Nodes represent the intersection points of all the arcs. Each node \(n \in \mathcal N\) serves as a boundary to all of its adjacent arcs and hence intertwines and connects them by the so-called \emph{flow balance equation}, which is a \emph{Kirchhoff typed law}
\begin{align}
    q^{(n)}(t) + \sum_{a \in \mathcal A_r[n]} 
    q_a(L, t) = \sum_{a \in \mathcal A_\ell[n]} q_a(0, t).
\end{align}
Here \(\mathcal A_\ell[n]\) is the set of all arcs (i.e., ordered tuples) \(a \in \mathcal A\) whose first node is \(n\) and \(\mathcal A_r[n]\) likewise is the set of all arcs \(a \in \mathcal A\) whose second node is \(n\), i.e.
\begin{align*} 
    a \in \mathcal A_\ell[n] &\iff \exists m \in \mathcal N: a = (n, m) \in \mathcal A, &
    a \in \mathcal A_r[n] &\iff \exists m \in \mathcal N: a = (m, n) \in \mathcal A.
\end{align*}
Furthermore, \(q^{(n)}\) is the gas inflow or outflow to the network, and therefore \(q^{(n)} = 0\) for most nodes.
However, if \(n\) is an entry or exit, a time-dependent piecewise constant or piecewise linear flow-profile \(q^{(n)}:\, [t_0, T]\subseteq \R \to \R\) is assumed to be provided. 
A second kind of Kirchhoff law hosted at nodes mediates boundary pressures from adjacent arcs in the following sense
\begin{align}
    \forall a \in \mathcal A_\ell[n]:\, p_a(0, t) &= p^{(n)}(t), & \forall a \in \mathcal A_r[n]:\, p_a(L, t) &= p^{(n)}(t),
\end{align}
where \(p^{(n)}(t)\) is the so-called \emph{node pressure}. 
Hence, this Kirchhoff law forces adjacent boundary pressures to coalesce at all times and associates the mutual value to be the pressure of the node. 

Arcs represent mostly \emph{pipes} but may also stand for other network elements such as active elements, like \emph{valves}, \emph{regulators}, and \emph{compressors}.
For each arc \(a = (n_\ell, n_r) \in \mathcal A\), the flow \(q(x, t) \equiv q_a(x, t)\) and pressure \(p(x, t) \equiv p_a(x, t)\) on both its ends are of particular interest and, therefore, will be abbreviated by \(q_\ell(t) \equiv q(0, t)\), \(q_r(t) \equiv q(L, t)\), \(p_\ell(t) \equiv p(0, t)\), and \(p_r(t) \equiv p(L, t)\). 
Correspondingly, the same holds true for the time derivatives, such that \(\dot p_\ell(t) = \dot p(0, t)\), \(\dot p_r(t) = \dot p(L, t)\), \(\dot q_\ell(t) = \dot q(0, t)\), and \(\dot q_r(t) = \dot q(L, t)\).

\subsection{On Spatial Discretizations of the Pipe Equations}
The literature offers quite a variety of potential spatial discretizations for the Equations~\eqref{eqn:pipe_pde}. 
We will pick two schemes for our presentation and numerical experiments. 
The first discretization in our listing originates as discretization for pipes in water networks \citep{JansenHuckTischendorf} and also has been studied in-depth in \citep{Huck2018Perturbation} for gas networks in advance. We will refer to it as \emph{left-right}-discretization
\begin{subequations}\label{eqn:pipe_ode_LR}
\begin{align}
    \dot p_r &= -\left[\frac{z^2(p_r)}{z(p_r) - p_r\cdot \partial_p z(p_r)}\right]\kappa \frac{q_r - q_\ell}L, \\
    0 &= A\frac{p_r - p_\ell}L + \frac{\lambda}{2D}\kappa z(p_\ell) \frac{q_\ell|q_\ell|}{p_\ell} + \delta\!h \frac{g}{\kappa} \frac{p_\ell}{z(p_\ell)}.
\end{align}
\end{subequations}

The second scheme, or \emph{implicit-box}-scheme, has been analyzed within \citep{KolLanBal2010}. This particular scheme actually combines a spatial and a time discretization altogether. Dissecting the spatial discretization relates to the trapezoidal as observed in \citep{TomSThesis} and reads
\begin{subequations}\label{eqn:pipe_ode_trap}
\begin{align}
    \dot p_r + \dot p_\ell &= -\left[\frac{z^2(p_\ell)}{z(p_\ell) - p_\ell\cdot \partial_p z(p_\ell)} + \frac{z^2(p_r)}{z(p_r) - p_r\cdot \partial_p z(p_r)}\right]\kappa\frac{q_r - q_\ell}L, \\
    0 &= A\frac{p_r - p_\ell}L + \frac{\lambda}{2 D}\kappa \left[z(p_\ell)\frac{q_\ell|q_\ell|}{p_\ell} + z(p_r)\frac{q_r|q_r|}{p_r}\right] + \delta\!h \frac{g}{\kappa} \left[\frac{p_\ell}{z(p_\ell)} + \frac{p_r}{z(p_r)}\right].
\end{align}
\end{subequations}
We will refer to \eqref{eqn:pipe_ode_trap} as the \emph{trapezoidal}-discretized pipe equation. 

\section{Towards the Regulator -- A Target Value Model}
\label{sec:target_value_model}
Regulators can down-regulate or halt gas flow through themselves, gain granular control of their surrounding neighborhood, and ultimately down the network.
They can also mediate and exchange gas between areas operated at different pressure levels. 
Their operation is controlled by target values.
In this section, we will provide a detailed explanation of the underlying target value logic, the interactions between pipes and regulators but also establish notions as well as terminology.

\subsection{On the Relations of Flows and Pressures}
\label{sec:relation_flow_pressure}
Both regulators and compressors do manipulate the gas flow through themselves.
Due to the gas dynamic modeled within neighboring pipes, the pressures will react and change as a consequence. 
Thus, from their direct influence on gas flow, active elements also have an indirect but immediate influence on pressures. 
Elaborating on this thought, either discretization as introduced so far does approximate pressures and flows linearly along a pipe, i.e., 
\begin{align*}
    q(x, t) &\approx q_\ell + x\frac{q_r - q_\ell}L\qquad \rightsquigarrow\qquad \partial_x q \approx \frac{q_r - q_\ell}L, \\
    p(x, t) &\approx p_\ell + x\frac{p_r - p_\ell}L\qquad \rightsquigarrow\qquad \partial_x p \approx \frac{p_r - p_\ell}L.
\end{align*}
Thus, we find a common precursor as space continuous system of differential equations with algebraic constraints derived from \eqref{eqn:pipe_pde} as
\begin{subequations}\label{eqn:pipe_space_continuous_ode}
\begin{align}
    \dot p(x, t) &= -B_z(p(x, t))\kappa\frac{q_r - q_\ell}L, \label{eqn:continuity_space_continuous_ode} \\ 
    0 &= A\frac{p_r - p_\ell}L + \frac{\lambda}{2 D}\kappa z(p(x, t)) \frac{q(x, t)|q(x, t)|}{p(x, t)} + \delta\!h \frac{g}{\kappa} \frac{p(x, t)}{z(p(x, t))}, \label{eqn:momentum_space_continuous_ode} 
\end{align}
\end{subequations}
where we refer to \(B_z(p(x, t)) \equiv \left[z^2(p(x, t))/[z(p(x, t)) - p(x, t)\cdot \partial_p z(p(x, t))]\right]\) as \emph{bracket term}. 
Now provided that \(B_{z}(p) > 0\), we can conclude for \eqref{eqn:continuity_space_continuous_ode} that net-positive inflows imply positive time derivatives of pressures,
i.e.,
\begin{subequations}\label{eqn:sensitivities}
    \begin{align}
        q_\ell > q_r \quad \implies \quad \dot p > 0, \\
        \intertext{and likewise the other way around}
        q_\ell < q_r \quad \implies \quad \dot p < 0,
    \end{align}
\end{subequations}
for net-positive outflows.
Furthermore, \(B_z(p) > 0\) is indeed true if and only if \(z(p) > p\partial_p z(p)\), which is generally true for any constant and linear model of the real gas factor and also true for the quadratic formula of Papay on its recommended operational range.
Hence, we could take control over local pressures in a predictable manner by manipulating the exchanged flow between two serial pipes. This is the fundamental mechanism on which we intend to base our regulator modelings upon.

Thus in the sense of this paper, we consider a regulator (sometimes also referred to as control valve) \(\CV\) to be a sub-network in itself, involving \(1\) entry pipe \(\mathcal P_\ell^\CV\) followed by a so-called \emph{atomic regulator} \(\CV_{\mathrm{atomic}}\) followed by \(1\) more exit pipe \(\mathcal P_r^\CV\). Figure \ref{fig:typical_cv_configurations} depicts the described sub-network configuration.

\begin{figure}[h!]
    \centering
    \begin{tikzpicture}[circuit, circuit symbol unit=6pt]
        \coordinate (L) at (0, -2);
        \coordinate (L_mid_bot) at (3.0, -2);
        \coordinate (Mid) at (6.5, -2.2);
        \coordinate (R_mid_bot) at (10.0, -2);
        \coordinate (R) at (13, -2);
        
        \coordinate (L_L) at (0, 0);
        \coordinate (L_mid_L) at (1.5, 0.0);
        \coordinate (L_mid_R) at (3.5, 0.0);
        \coordinate (L_R) at (5, 0);
        \coordinate (LR_between) at (6.5, -0.2);
        \coordinate (R_L) at (8, 0);
        \coordinate (R_mid_L) at (9.5, 0.0);
        \coordinate (R_mid_R) at (11.5, 0.0);
        \coordinate (R_R) at (13, 0);
        
        \draw[dashed] (L) to (L_L);
        \draw[dashed] (R) to (R_R);
        
        \draw (L) to [cv={label=$\CV$}] (R);
        
        \node[fill=white, draw, circle, label=below:{$p_\ell$}] at (L) {$n_\ell$};
        \node[fill=white, draw, circle, label=below:{$p_r$}] at (R) {$n_r$};
        
        \node[label = below:{$q$}] at (Mid) {};
        
        \draw (L_L) to [pi={label=$\mathcal P_\ell^\CV$}] (L_R);
        \draw (L_R) to [cv={label=$\CV_{\mathrm{atomic}}$}] (R_L);
        \draw (R_L) to [pi={label=$\mathcal P_r^\CV$}] (R_R);
        
        \node[fill=white, draw, circle, label=above:{$p_\ell$}] at (L_L) {$n_\ell$};
        \node[fill=white, draw, circle, label=above:{$p^{(n_a)}$}] at (L_R) {$n_a$};
        \node[fill=white, draw, circle, label=above:{$p^{(n_b)}$}] at (R_L) {$n_b$};
        \node[fill=white, draw, circle, label=above:{$p_r$}] at (R_R) {$n_r$};
        
        \node[label = below:{$q_\ell$}] at (L_mid_bot) {};
        \node[label = below:{$q_\ell$}] at (L_mid_L) {};
        \node[label = below:{$q$}] at (L_mid_R) {};
        \node[label = below:{$q$}] at (LR_between) {};
        \node[label = below:{$q$}] at (R_mid_L) {};
        \node[label = below:{$q_r$}] at (R_mid_R) {};
        \node[label = below:{$q_r$}] at (R_mid_bot) {};
    \end{tikzpicture}
    \caption{Depiction of the hidden sub-network structure (above in the picture) behind a regulator or sometimes control valve \(\CV\) (below in the picture), including an atomic regulator \(\CV_{\mathrm{atomic}}\) as well as two internal pipes \(\mathcal P^\CV_\ell\) and \(\mathcal P^\CV_r\).}
    \label{fig:typical_cv_configurations}
\end{figure}
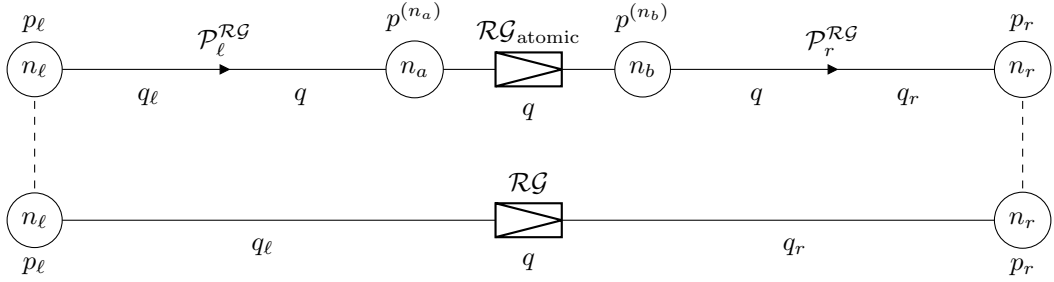

Now the atomic regulator centering between both hidden pipes is allowed to change the flow $q$ by increasing or decreasing its resistance. 
In other words, it exploits our observations in \eqref{eqn:sensitivities} which roughly translate into the following expected behavioral tendencies
\begin{align*}
  q \nearrow \qquad&\rightsquigarrow\qquad p_\ell \searrow p_r \nearrow\ , \\
  q \searrow \qquad&\rightsquigarrow\qquad p_\ell \nearrow p_r \searrow.
\end{align*}

\subsection{Regulators Controlling the Gas Flow} \label{sec:element_control}
\begin{figure}
  \centering
  \includegraphics[width=0.6\textwidth]{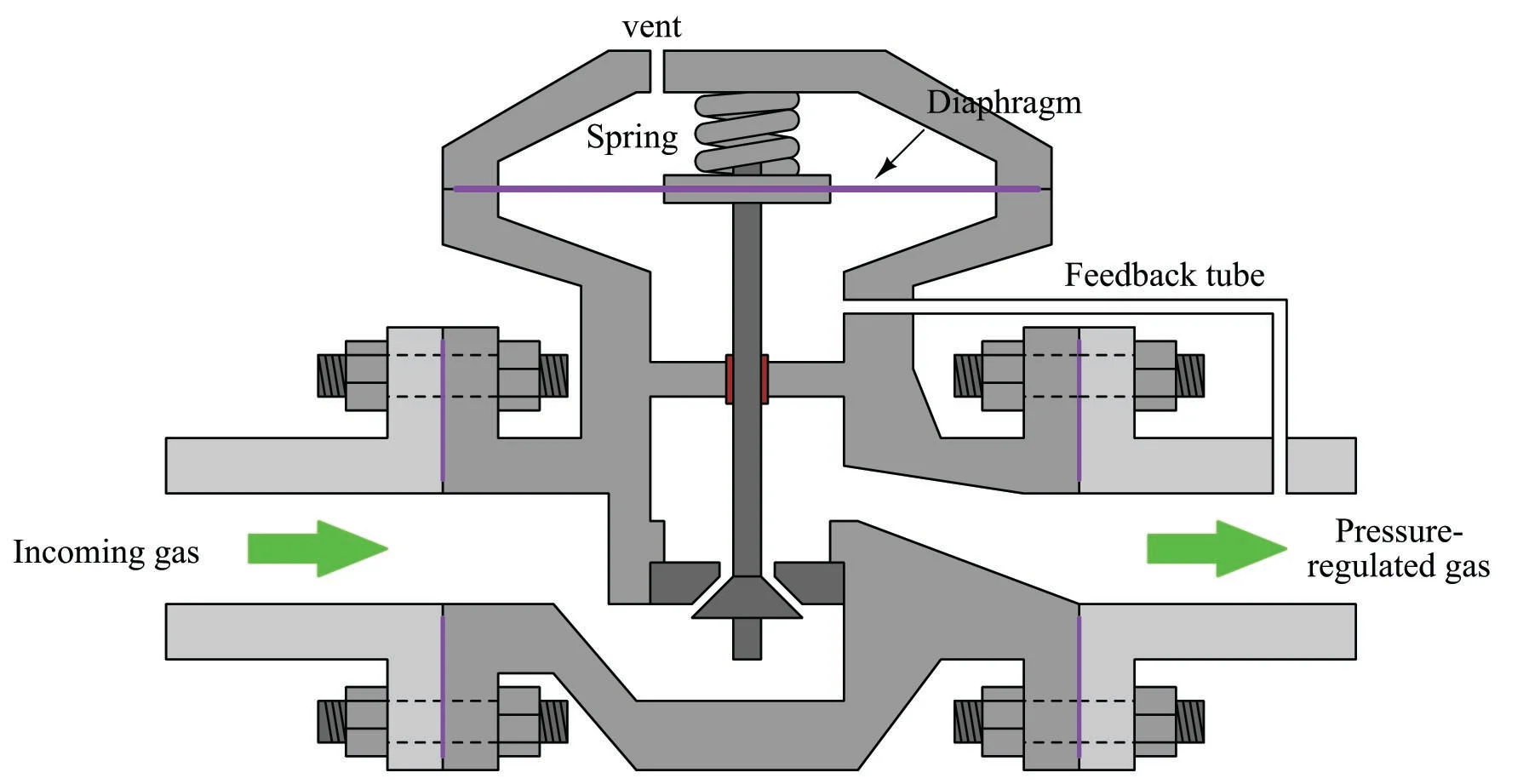}
  \caption[Visualization of a regulator controlling the right pressure]{Visualization of a regulator controlling the right pressure; picture created by \citep{Kup2019}, published under CC BY 4.0\protect\footnotemark}
  \label{fig:schematicRegulator}
\end{figure}
\footnotetext{\url{https://creativecommons.org/licenses/by/4.0/}}
The task of regulators in gas networks is to reduce the pressure along their orientation by decreasing the flow through the element.
This flow decrease is accomplished by reducing the regulator's opening degree $o\in[0,1]$, which can be interpreted as changing the diameter of the regulator, assuming its cross-sectional area is cylindrical, see~\citep{FugGeiGolMor2015}.
If a regulator is fully opened at $o=1$, it does not reduce the flow and $p_\ell=p_r$ holds.
At $o=0$, a regulator is fully closed, which disconnects the network.
Hence, $q=0$ and the pressures are decoupled.
Values in between induce artificial resistance to reduce the flow and therefore create a pressure decreases in flow direction as derived in Section~\ref{sec:relation_flow_pressure}, i.e., $p_\ell \geq p_r$.
A schematic visualization of the principle operation of a regulator is depicted in Figure~\ref{fig:schematicRegulator}. It originates from \citep{Kup2019}, where more insights on various technical details and further information on other gas network elements are provided, too.

\subsection{Target Value Control}\label{sec:description_target_value_control}
For gas network operators, i.e., dispatchers, adjusting the opening degree of single elements by hand is far too involved and would require continuous observation.
Instead, regulators have automated systems fed by desired ranges for the \(3\) local intrinsic quantities \(p_\ell, p_r, q_m\). The \(6\) defining bounds for these ranges \(\MINPIN\), \(\MAXPIN\), \(\MINPOUT\), \(\MAXPOUT\), \(\MINQ\), \(\MAXQ\) are called \emph{target values} (and sometimes also referred to as set-point values).
An overview of the \(6\) types of target values and their consequential influence is listed in Table~\ref{tab:tvOverview}.

\begin{table}[ht]
    \centering
    \begin{tabular}{llcrc}
    Target Value & Symbol & Priority \funcTvPrio & imposed bounds & change if violated 
    \medskip\\
    minimal left pressure & \MINPIN  & 4 & $\MINPIN  \leq \mathmakebox[1cm][l]{p_\ell}$ & closing \\
    maximal right pressure & \MAXPOUT & 4 & $\MAXPOUT \geq \mathmakebox[1cm][l]{p_r}$    & closing \\
    maximal left pressure & \MAXPIN  & 3 & $\MAXPIN  \geq \mathmakebox[1cm][l]{p_\ell}$ & opening \\
    minimal right pressure & \MINPOUT & 3 & $\MINPOUT \leq \mathmakebox[1cm][l]{p_r}$    & opening \\
    maximal flow              & \MAXQ    & 2 & $\MAXQ    \geq \mathmakebox[1cm][l]{q}$      & closing \\
    minimal flow              & \MINQ    & 1 & $\MINQ    \leq \mathmakebox[1cm][l]{q}$      & opening
    \end{tabular}
    \caption{Overview of all types of target values sorted by priority from high at the top to low at the bottom. We also listed the implied bound on the corresponding quantity of the regulator as well as the direction of change of the regulator's opening degree, which is triggered if the target value is violated.}
    \label{tab:tvOverview}
\end{table}

If a target value imposed bound is violated by the state values (we then also refer to the target value as being violated), the element adjusts according to Table~\ref{tab:tvOverview} to prevent or reduce the violation as much as possible.
To resolve conflicts of multiple violated target values demanding opposing changes, the target value types are prioritized by $\funcTvPrio(\tv)$ for the target values $\tv$ from top to bottom.
In general, pressure target values and those reducing the opening degree and thereby the flow have higher priorities.
Some target values may share their priority value, but only if they influence the element's control in the same way.

Note that it is possible to force the regulator to be fully opened or fully closed by using certain target value combinations, for example, by setting
\begin{center}
\begin{align*}
    \MINPIN = \MINPOUT = 0, \quad \MAXPIN = \MAXPOUT = \MINQ = \MAXQ = \infty & \quad \text{for the \textbf{open} mode or} \\
    \MINPIN = \MAXPIN = \infty, \quad \MINPOUT = \MAXPOUT = \MINQ = \MAXQ = 0 & \quad \text{for the \textbf{closed} mode}.
\end{align*}
\end{center}

\paragraph{The Regulator’s Check Valve}
A regulator contains a built-in element that prevents flow against its orientation, i.e., $q\geq 0$.
This element is called a check valve or flap trap and automatically closes the regulator in case the right pressure rises above the left one, i.e. $p_\ell < p_r$.
This action has a higher priority than all the target values and therefore happens independently of these.
If the regulator is not closed by the check valve, $p_\ell \geq p_r$ holds.

\paragraph{Infinite Minimal Flow Target Value}
For most elements, the lowest priority target value \MINQ is fixed to $\infty$, which effectively is the same as choosing \(\MINQ = \MAXQ \le \infty\), in which case we may also denote \(\SETQ \equiv \MINQ = \MAXQ\).
Regulators were actually \(\infty \neq \MINQ < \MAXQ\) are called ``flow band regulators'' or simply ``band regulators''.
They keep their current degree of openness and do not react to changes in the flow as long $q$ stays within \(\MINQ \le q \le \MAXQ\) and provided that all other target values are satisfied.
This behavior would increase the overall complexity of the models used for the simulation and optimization of target values tremendously and is out of this paper's scope.

\paragraph{Target Value Existence}
Not all target values are always in use.
This has to be taken into account by the models described below, for example by prohibiting their violation by setting maximal target values to $\infty$ and minimal target values to \(0\).

However, we assume that the fixed target value \MINQ always exists and in addition at least one of the closing target values.
Otherwise, the element's control could only be changed in one direction.

\subsection{Target Values for Compressors}\label{sec:compressors}
In addition to regulators, also active compressors are controlled by target values.
While compressors \emph{increase} the pressure in flow direction, they influence their flow throughput to achieve these local pressure changes similar to regulators.
For the two most common types of compressors, the flow changes are realized by controlling the compressor's rotational speed, see~\citep{FugGeiGolMor2015, SchAssmBurHum2017} for a more detailed explanation.
Due to the similarities, the target value description for compressors is very close to the one of regulators.
However, it is out of this paper's scope.

\section{Modeling Regulators for Dynamic Simulations}\label{sec:simulationModel}

For the derivation of regulator models for simulations, we will follow \citep{TomSThesis}.
Target values are no bounds in a classical sense. 
Instead, they influence and guide the behavior of a regulator over time.
Among ordinary differential equations, we find
\begin{align}
    \dot x(t) = \alpha[c(t) - x(t)],\quad \text{where}\quad x(t_0) = x_0 \label{eqn:prototypeCV_ODE}
\end{align}
as a prototype model. We can utilize \eqref{eqn:prototypeCV_ODE} to model the different responses of regulators to answer various kinds of imminent or present target value violations. The solution to Equation~\eqref{eqn:prototypeCV_ODE} is
\begin{align*}
    x(t) = \exp(t_0 - t)^\alpha\cdot x_0 + \alpha\int_{t_0}^t \exp(s - t)^\alpha\cdot c(s)\, \mathrm ds
\end{align*}
and always tends towards \(c(t)\). In some sense, we may describe the behavior of \(x\) as chasing \(c(t)\), or we may see \(x(t)\) as a leaky or lazy smoothing of \(c(t)\). In case \(c(t) = c_0\) is constant, then \(x(t)\) would also converge. With the parameter \(\alpha > 0\), we can adjust how loose or tight \(x(t)\) will follow \(c(t)\). In the limit case \(\dot x/\alpha \xrightarrow{\alpha \to \infty} 0\), we find \(x(t) = c(t)\). 

We adapt this model to realize a regulator following its target values. We may substitute \(x\) for the flow \(q\) and replace \(c(t)\) with the target value for flows and find
\begin{align}
    \dot q = \alpha[\SETQ - q], \label{eqn:cv_derivation_stage_i}
\end{align}
given \(\SETQ = \MINQ = \MAXQ\).
On top of \eqref{eqn:cv_derivation_stage_i}, we may add the check valve \(\dot q = \alpha\max(0, -q)\) and non-compressing behavior \(\dot q = \alpha\min(p_\ell - p_r, 0)\) and find
\begin{align}
     \dot q = \alpha\max(-q, \min(p_\ell - p_r, \SETQ - q)), \label{eqn:cv_derivation_stage_ii}
\end{align}
which is a nesting of \(\max-\min\)-comparisons.
This nesting is necessary and represents the prioritization of the already integrated features.
Also, we find sub-modelings for all remaining target values
\begin{align*}
    \dot q &= \alpha\min(0, \MAXPOUT - p_r), & \dot q &= \alpha\min(0, p_r - \MINPOUT), &
    \dot q &= \alpha\min(0, \MAXPIN - p_\ell), & \dot q &= \alpha\min(0, p_\ell - \MINPIN),
\end{align*}
where we can exploit our knowledge \eqref{eqn:sensitivities} of relations between flow and adjacent pressures of neighboring pipes \(\mathcal P_\ell^{\CV}\), \(\mathcal P_r^{\CV}\) of Figure \ref{fig:typical_cv_configurations}. 
Finally, we combine all single-aspect models within the nesting of \eqref{eqn:cv_derivation_stage_ii} and end up with the full regulator model
\begin{align}
    \begin{aligned}
        \dot q = \alpha\max(-q, \min(&p_\ell - \max(\MINPIN, p_r), \\
        &\min(\MAXPOUT, p_\ell) - p_r, \\
        &\max(\SETQ - q, p_\ell - \MAXPIN, \MINPOUT - p_r))),
    \end{aligned}\label{eqn:cv_full_model}
\end{align}
for \(\SETQ = \MINQ = \MAXQ\). Two possible modifications of \eqref{eqn:cv_full_model} are
\begin{align}
    \begin{aligned}
        \dot p_r + \dot q - \dot p_\ell = \alpha\max(-q, \min(&p_\ell - \max(\MINPIN, p_r), \\
        &\min(\MAXPOUT, p_\ell) - p_r, \\
        &\max(\SETQ - q, p_\ell - \MAXPIN, \MINPOUT - p_r))),
    \end{aligned}\tag{\ref{eqn:cv_full_model}.i}\label{eqn:cv_full_model_b}
\end{align}
or the limit system \(\dot q/\alpha \xrightarrow{\alpha \to \infty} 0\)
\begin{align}
    \begin{aligned}
        0 = \max(-q, \min(&p_\ell - \max(\MINPIN, p_r), \\
        &\min(\MAXPOUT, p_\ell) - p_r, \\
        &\max(\SETQ - q, p_\ell - \MAXPIN, \MINPOUT - p_r))).
    \end{aligned}\tag{\ref{eqn:cv_full_model}.ii}\label{eqn:cv_full_model_c}
\end{align}

\section{Modeling Regulators for Discrete Optimization} \label{sec:optimizationModel}
To be able to optimize over the regulator's target values in the context of time discretized transient gas network operation, we aim for a model suitable for discrete optimization solvers.
Our objective is to minimize the number of changes in the target values such that the control decisions induced by them are feasible to fulfill the demands of the network, which are usually given at the entries and exits.

For improved readability, we use \emph{indicator constraints} of the form
\begin{equation*}
  y = b \quad\rightarrow\quad a^T x \leq a_0 \qquad y,b\in\{0,1\}, x,a\in\setReals^i, a_0\in\setReals
\end{equation*}
stating that the constraint $a^T x \leq a_0$ for a set of variables $x = \{x_1, \dots, x_i\}$ is active if the binary variable $y$ attains the value $b$.
If $x$ is bounded, these can be reformulated using linear constraints, see for example~\citep{BonLodTraWie2015}.

\subsection{A Basic Regulator Model} \label{sec:optBasicRegulatorModel}
We use the following model describing the general behavior of a regulator $a=(\ell,r)$ from the literature, see for example \citep{KocHilPfeSch2015,HenAndHopTur2021}:
\begin{align*}
  \pressI{\ell}, \pressI{r} \in \setRealsNonNeg, \quad \mFlowI{a} &\in \setReals && \text{operation point variables} \\
  \modeAc{a}, \modeOp{a}, \modeCl{a}, \modeCv{a} &\in\{0,1\} && \text{mode variables}\\
\end{align*}
\begin{subequations}\label{eq:opt_basic_regulator_model}
\begin{align}
    \modeAc{a} + \modeOp{a} + \modeCl{a} + \modeCv{a}  &= 1\\
    \modeOp{a} + \modeCv{a} = 1 \quad\rightarrow\quad \pressI{\ell} &\leq \pressI{r} \\
    \modeAc{a} + \modeOp{a} = 1 \quad\rightarrow\quad \pressI{\ell} &\geq \pressI{r} \\
    \modeCl{a} + \modeCv{a} = 1 \quad\rightarrow\quad \mathmakebox[\widthof{\pressI{\ell}}][r]{\mFlowI{a}} &\leq 0 \\
    0 \leq \mFlowI{a} \enspace\quad &
\end{align}
\end{subequations}
The model features the mode variables \modeOp{a} for an open regulator, \modeCl{a} for a closed regulator, and \modeAc{a} for an active regulator with opening degree in $(0,1)$, which has also been used in \citep{HenAndHopTur2021}.
We add the mode $\modeCv{a}$ for a regulator closed by the internal check valve in case of $\pressI{\ell} \leq \pressI{r}$.
In this case, the target values to not influence the regulator according to Section~\ref{sec:description_target_value_control}.
Note that the open mode is often replaced by a ``bypass'' mode in the literature, representing a bypassing network path that also allows flow in the backwards direction, which is not included in our regulator model.

\subsection{Stable-Pushing Target Value Combinations}\label{sec:optStablePushingCombis}
In preparation for the model construction, we determine for each target value the cases in which its implied bound has to be obeyed.
While in reality, the adjustments of active elements caused by violated target values happen with a certain delay, 
we make the following assumption for our target value model:
\begin{assumption}\label{ass:perfectControl}
The control of each active element, i.e., the opening degree adjustment for regulators, reacts to changing target values or operation point conditions \emph{immediately} and with \emph{perfect precision} according to the given list of target value priorities \funcTvPrio.
Hence, we can assume the control is always fully adjusted to the given network situation and target values.
\end{assumption}
Note that this is equivalent to the simulation model \eqref{eqn:cv_full_model_c} assuming $\alpha \to \infty$.
We observe the following:
\begin{observation}\label{obs:oneViolatedTVOfHighestPrio}
Since \MINQ always exists and is violated, there always is at least one violated target value.
From these target values, we call the one with the highest priority \funcTvPrio the \emph{pushing} target value.
\end{observation}
Note that there might be multiple violated target values of highest priority.
However, these always have the same \emph{direction}, i.e., either opening or closing altogether, see Table~\ref{tab:tvOverview}.

Now, we can deduce the following theorem.
\begin{theorem}\label{thm:regulatorStateCharacterization}
If a regulator $a$ is not closed due to its check valve, i.e., $\pressI{\ell} < \pressI{r}$, one of the following statements hold: 
\begin{enumerate}[1)]
\item There is at least one target value $\tv$, called the \emph{stable} target value, with higher priority than the pushing target value as well as opposing direction in terms of change in case of violation, such that $\tv$ is equal to its corresponding operation point quantity.
\item The regulator $a$ is fully open, i.e., in open mode, and the pushing target value is an opening target value.
\item The regulator $a$ fully closed, i.e., in closed mode, and the pushing target value is a closing target value.
\end{enumerate}
\end{theorem}
\begin{proof}
Observation~\ref{obs:oneViolatedTVOfHighestPrio} states the existence of the pushing target value as a violated target value of highest priority at any given point in time.
According to the description given in Section~\ref{sec:description_target_value_control}, the regulator tries to adjust the opening degree to reduce the corresponding violation.
However, according to Assumption~\ref{ass:perfectControl}, the opening degree is already perfectly adjusted and therefore does not change.
The two possible reasons are:
\begin{enumerate}[a)]
  \item The regulator cannot adjust the opening degree in the desired direction since it is already fully opened or closed.
  \item There is some higher priority target value that would be violated if the opening degree would be changed by any amount in the desired direction due to the induced operation point changes.
\end{enumerate}
For case a), either statement 2) or 3) holds, depending on the direction of the pushing target value's opening degree change.
For case b), the higher priority target value has to have the opposite opening degree change direction compared to the pushing target value.
Furthermore, the current target value must be equal to the corresponding operation point quantity since any change in the opening degree, no matter how small, would cause a violation.
Hence, statement 1) is true.
\end{proof}
Note that while the situations of 2) and 3) are coupled to a specific regulator mode, this is not the case for 1), where the two opposing target values may fix the opening degree at a position that just happens to be fully open or fully closed.

Using Theorem~\ref{thm:regulatorStateCharacterization}, we are able to compile a complete list of possible configurations in terms of stable and pushing target value combinations, which cover all possible states of a fully adjusted regulator according to Assumption~\ref{ass:perfectControl}.
It is displayed in Table~\ref{tab:stablePushingTVCombi}.
Each row represents either a possible stable-pushing target value combination with opposite directions or a single pushing target value without a stable target value counterpart. 
Since the pushing target value is by definition the highest priority target value that is violated, all bounds implied by target values having a higher priority have to be satisfied.
On the other side, all constraints of target values with priority lower than or equal to the pushing target value are irrelevant for that stable-pushing combination.
\begin{table}[t]
  \centering
  \setlength{\tabcolsep}{6pt}
  \renewcommand{\arraystretch}{1}
  \definecolor{finee}{RGB}{154, 203, 255}
  \definecolor{wrong}{RGB}{204,  43,  43}
  \definecolor{stabl}{RGB}{255, 188,  73}
  
  \newcommand{\styleFinee}{\cellcolor{finee}}
  \newcommand{\styleWrong}{\cellcolor{wrong}\color{white}}
  \newcommand{\styleStabl}{\cellcolor{stabl}}
  
  \begin{tabular}{lcccccccc}
    Mode & Stable & Pushing & \MINPIN & \MAXPOUT & \MAXPIN & \MINPOUT & \MAXQ & \MINQ
    \medskip \\
    any & \MAXQ    & \MINQ    & \styleFinee  $\min$    & \styleFinee  $\max$    & \styleFinee  $\max$    & \styleFinee $\min$     & \styleStabl $\hat{q}$ & \styleWrong $\max$ \\
    any & \MINPIN  & \MAXPIN  & \styleStabl $\hat{p}$ & \styleFinee  $\max$    & \styleWrong $\min$    & $\ast$                           & $\ast$                           & $\ast$                            \\
    any & \MINPIN  & \MINPOUT & \styleStabl $\hat{p}$ & \styleFinee  $\max$    & $\ast$                           & \styleWrong $\max$    & $\ast$                           & $\ast$                            \\
    any & \MINPIN  & \MINQ    & \styleStabl $\hat{p}$ & \styleFinee  $\max$    & \styleFinee  $\max$    & \styleFinee  $\min$    & \styleFinee  $\max$    & \styleWrong $\max$ \\
    any & \MAXPIN  & \MAXQ    & \styleFinee  $\min$    & \styleFinee  $\max$    & \styleStabl $\hat{p}$ & \styleFinee  $\min$    & \styleWrong $\min$    & $\ast$                            \\
    any & \MAXPOUT & \MAXPIN  & \styleFinee  $\min$    & \styleStabl $\hat{p}$ & \styleWrong $\min$    & $\ast$                           & $\ast$                           & $\ast$                            \\
    any & \MAXPOUT & \MINPOUT & \styleFinee  $\min$    & \styleStabl $\hat{p}$ & $\ast$                           & \styleWrong $\max$    & $\ast$                           & $\ast$                            \\
    any & \MAXPOUT & \MINQ    & \styleFinee  $\min$    & \styleStabl $\hat{p}$ & \styleFinee  $\max$    & \styleFinee  $\min$    & \styleFinee  $\max$    & \styleWrong $\max$\\
    any & \MINPOUT & \MAXQ    & \styleFinee  $\min$    & \styleFinee  $\max$    & \styleFinee  $\max$    & \styleStabl $\hat{p}$ & \styleWrong $\min$    & $\ast$                            \\
    closed & \quad$-$ & \MAXQ    & \styleFinee  $\min$    & \styleFinee  $\max$    & \styleFinee  $\max$    & \styleFinee $\min$     & \styleWrong $\min$    & $\ast$                            \\
    closed & \quad$-$ & \MINPIN  & \styleWrong $\max$    & $\ast$                           & $\ast$                           & $\ast$                           & $\ast$                           & $\ast$                            \\
    closed & \quad$-$ & \MAXPOUT & $\ast$                           & \styleWrong $\min$    & $\ast$                           & $\ast$                           & $\ast$                           & $\ast$                            \\
    open & \quad$-$ & \MINQ    & \styleFinee  $\min$    & \styleFinee  $\max$    & \styleFinee  $\max$    & \styleFinee $\min$     & \styleFinee  $\max$    & \styleWrong $\max$ \\
    open & \quad$-$ & \MAXPIN  & \styleFinee  $\min$    & \styleFinee  $\max$    & \styleWrong $\min$    & $\ast$                           & $\ast$                           & $\ast$                             \\
    open & \quad$-$ & \MINPOUT & \styleFinee  $\min$    & \styleFinee  $\max$    & $\ast$                           & \styleWrong $\max$    & $\ast$                           & $\ast$                             
  \end{tabular}
  \caption{The table shows all possible stable-pushing target value combinations, one for each line.
  In addition to the stable target value, the pushing target value, and the possible regulator modes, it shows the status of the bound for each target value encoded by the color of the cell: Blue stands for satisfied bounds, red for violated bounds, and orange for tight bounds in which the target value is equal to the operation point value.
  The text in the cells describes the necessary target value choice for each combination:
  The stable target value is marked by the quantity symbol combined with a ``\^{}''.
  For the violated and satisfied target values, the two possible values are ``min'' and ``max''.
  The value ``min'' encodes a target value that needs to be \emph{smaller} than the state value, while ``max'' represents a target value that needs to be \emph{larger} than the state value.
  The cells without coloring and filled with an asterisk do not have any relevant condition for the given row since the pushing target value has the same or a higher priority \funcTvPrio.}
  \label{tab:stablePushingTVCombi}
\end{table}

Note that in the case of non-existing target value types, the corresponding lines with non-existing stable or pushing target values, as well as the columns of the missing target value types, would be removed.

\subsection{A Mixed Integer Linear Formulation}\label{sec:optMILPModel}
We now present our mixed integer linear programming (MILP) target value model using the characterization derived in the previous section.
However, in contrast to the overview presented in Table~\ref{tab:stablePushingTVCombi}, we will only allow the usage of the active mode for all target value combinations featuring a stable target value.
This is based on the fact that we only use $\leq$ and $\geq$ relations to represent violated bounds since strict inequalities cannot be used in the MILP context.
As a result we formulate the following theorem:
\begin{theorem}\label{thm:opt_only_active_mode_for_stable}
Given a stable-pushing target value combination $c_1$ as well as a set of target values and state values satisfying the constraints implied by $c_1$.
Then these values also satisfy the constraints for an open mode combination $c_2$ and a closed mode combination $c_3$.
The two different pushing target values of $c_2$ and $c_3$ are equal to the stable and the pushing target value of $c_1$.
\end{theorem}
\begin{proof}
Since we only use $\leq$ and $\geq$ to represent violated bound constraints, the stable target value of $c_1$ can be alternatively interpreted as a violation as well as a fulfillment of the corresponding target value bound.
Hence, the values satisfy the constraints of that non-active mode combination $c_4$, which has the stable target value of $c_1$ as pushing target value, i.e., that violated target value of highest priority.
The same is true for that non-active mode combination $c_5$, which has the pushing target value of $c_1$ as pushing target value without having an explicit stable target value.
Since the stable and pushing target values of $c_1$ have opposite directions, one is a closing target value, and one is an opening one.
Hence, one of the combinations $c_4$ and $c_5$ has the open mode and is the aforementioned combination $c_2$, and the other has the closed mode and is $c_3$.
\end{proof}
As a result of Theorem~\ref{thm:opt_only_active_mode_for_stable}, we assume that the regulator is in active mode iff it has a stable target value since all the open and closed mode cases \emph{with} stable target value are covered by corresponding combinations \emph{without} a stable target value.
This enables us to formulate the following MILP model for a target value control of a single regulator $a=(\ell,r)$ at time $t$.
As the model extends the basic regulator model \eqref{eq:opt_basic_regulator_model}, we will only state the additional variables and constraints.
First, we introduce some sets:
\begin{center}
\begin{tabular}{lcll}
  \setTargetValues & $:=$ & $\{\MINPIN, \MAXPIN, \MINPOUT, \MAXPOUT, \MINQ, \MAXQ \}$ & All target values \\ 
  \setTargetValuesI{a} &&& All target values \emph{existing} for regulator $a$ \\
  \setStaPushActive    & $:=$ & $\{(\MAXQ, \MINQ), (\MINPIN,\MAXPIN), (\MINPIN,\MINPOUT),  $ & All stable-pushing target value combinations for active \\
                       &      & $(\MINPIN,\MINQ), (\MAXPIN, \MAXQ), (\MAXPOUT, \MAXPIN), $ &  \\
                       &      & $(\MAXPOUT, \MINPOUT), (\MAXPOUT, \MINQ), (\MINPOUT, \MAXQ)\}$ \\
  \setStaPushOpen    & $:=$ & $\{ (-,  \MAXPIN), (-,  \MINPOUT), (-,  \MINQ)\}$ & All stable-pushing target value combinations for open \\
  \setStaPushClosed    & $:=$ & $\{ (-,  \MINPIN), (-,  \MAXPOUT), (-,  \MAXQ)\}$ & All stable-pushing target value combinations for closed \\
  \setStaPushActiveI{a} &&& All \emph{existing} combinations for active for regulator $a$\\
  \setStaPushOpenI{a} &&& All \emph{existing} combinations for open for regulator $a$\\
  \setStaPushClosedI{a} &&& All \emph{existing} combinations for closed for regulator $a$\\
  \setStaPushI{a}       & $:=$ & $\setStaPushActiveI{a} \cup \setStaPushOpenI{a} \cup \setStaPushClosedI{a}$ & All \emph{existing} combinations for regulator $a$
\end{tabular}
\end{center}

Now, we can specify the new variables as
\begin{align*}
  \tvI{x,a,t} &\in\left[\lbTVI{x,a},\ubTVI{x,a}\right] & \forall x &\in \setTargetValuesI{a} && \text{target values}\\
  \tvCombiI{y,a,t}  &\in\{0,1\} & \forall y &\in \setStaPushI{a} && \text{stable-pushing target value combinations}\\
  \objVarI{x,a,t}  &\in\{0,1\} & \forall x &\in \setTargetValuesI{a} && \text{indicator for target values changes}
\end{align*}
The \tv variables represent the target values chosen by the model.
They can be set inside the variable  bounds \lbTVI{x,a} and \ubTVI{x,a} for each target value $x$.
These are given for each individual regulator $a$ and can considerably restrict the feasible operating range of the regulator.
If, for example, the minimal left pressure target value \MINPIN has a lower bound value of $\SInumber{40}{bar}$, then the regulator has to operate at a left pressure value of at least $\SInumber{40}{bar}$ in active mode and open mode, since according to Table~\ref{tab:stablePushingTVCombi}, the target value \MINPIN can only be violated for the closed mode or if the regulator is closed due to the check valve.
The variables \tvCombi represents the stable-pushing target value combination chosen for each regulator $a$.
Finally, the \objVar variables indicate a change in the target value variables \tv between the values of the previous and the current time point.

For the following set of constraints, we introduce the function $o$ mapping from a given type of target value to the corresponding point of operation:
\begin{align*}
    o_t(x) :=  
    \begin{cases}
        \pressI{\ell,t} & \text{for } x \in \{\MINPIN, \MAXPIN\} \\
        \pressI{r,t}    & \text{for } x \in \{\MINPOUT, \MAXPOUT\} \\
        \mFlowI{a,t}     & \text{for } x \in \{\MINQ, \MAXQ\}
\end{cases}
\end{align*}
In addition, we use \funcTvPrio as a function from target value types to the priorities given in Table~\ref{tab:tvOverview} above.
Also, we introduce the parameter $\varepsilon$, which represents a relative tolerance used in the equality constraint of the stable target value.
As explained at the beginning of Section~\ref{sec:optStablePushingCombis}, the adjustment of the regulator's control to changing target values or operating point conditions happens with a certain delay in reality.
Hence, there usually is an offset between the used target values and the current operation point.
To better reflect this behavior and prevent too many target value changes caused by minor adjustments of the stable target value, we introduced the relative $\varepsilon$ value.

The constraints for our target value model are then given as
\begin{subequations}\label{eq:opt_tv_model}
\begin{align}
    \modeX{a,t} &= \sum_{y\in\setStaPushXI{a}} \tvCombiI{y,a,t} \quad \forall \mathrm{x}\in\{\mathrm{ac}, \mathrm{by}, \mathrm{cl}\} \label{eq:combiChoicePerMode}\\
    \sum_{y=(x,v) \in \setStaPushI{a}} \tvCombiI{y,a,t} = 1 \quad\rightarrow\quad &
    - \varepsilon o_t(x) \leq \tvI{x,a,t} - o_t(x) \leq \varepsilon o_t(x)
    \quad \forall x \in \setTargetValuesI{a} \label{eq:combiTVConnectionStable}\\
    \sum_{y=(u,x) \in \setStaPushI{a}} \tvCombiI{y,a,t} = 1 \quad\rightarrow\quad &
    \begin{cases}
    \tvI{x,a,t} \geq o_t(x) & \text{if } x \in \{\MINPIN, \MINPOUT, \MINQ\} \\
    \tvI{x,a,t} \leq o_t(x) & \text{if } x \in \{\MAXPIN, \MAXPOUT, \MAXQ\}
    \end{cases}
    \quad \forall x \in \setTargetValuesI{a} \label{eq:combiTVConnectionViolated}\\
    \sum_{y=(u,v) \in \setStaPushI{a}, x\neq u \,\land\, \funcTvPrioI{x} > \funcTvPrioI{v}} \tvCombiI{y,a,t} = 1 \quad\rightarrow\quad &
    \begin{cases}
    \tvI{x,a,t} \leq o_t(x) & \text{if } x \in \{\MINPIN, \MINPOUT, \MINQ\} \\
    \tvI{x,a,t} \geq o_t(x) & \text{if } x \in \{\MAXPIN, \MAXPOUT, \MAXQ\}
    \end{cases}
    \quad \forall x \in \setTargetValuesI{a} \label{eq:combiTVConnectionSatisfied}\\
    \objVar_{x,a,t} = 0 \quad\rightarrow\quad & \tvI{x,a,t} = \tvI{x,a,t-1} \quad \forall x \in \setTargetValuesI{a} \label{eq:defTVchangeVars}
\end{align}
\end{subequations}
Equation \eqref{eq:combiChoicePerMode} ensures the choice of exactly one stable-pushing target value combination fitting to the current mode.
Note that for the check-valve-closed mode, no combination is chosen.
Constraints \eqref{eq:combiTVConnectionStable}, \eqref{eq:combiTVConnectionViolated}, and \eqref{eq:combiTVConnectionSatisfied} establish the consequences of the choice of a stable-pushing target value combination and are based on the illustration of Table~\ref{tab:stablePushingTVCombi}.
Structure-wise these constraints collect for a given target value type \tv all stable-pushing target value combinations in which this type of target value is the stable one for \eqref{eq:combiTVConnectionStable}, is violated for \eqref{eq:combiTVConnectionViolated}, or is satisfied for \eqref{eq:combiTVConnectionSatisfied}.
If one of these combinations is chosen, the corresponding relations between the target value and the corresponding operation point are enforced.
Finally, constraint \eqref{eq:defTVchangeVars} defines the target value change variables, forcing the actual target value variable $\tvI{x,a,t}$ to keep their value if the change variable does not indicate a change, i.e., is zero.

As objective function, we minimize the number of target value changes:
\begin{equation*}
  \min \sum_{a,t} \sum_{x \in \setTargetValuesI{a}} \objVarI{x,a,t}.
\end{equation*}

Note that \citep{PfeFueGeiGei2014,Schmidt2015,KocHilPfeSch2015} introduced a model for regulators without remote access that have a set-point control for the maximum right pressure with value $\press^\mathrm{set}$.
Depending on whether the right pressure is larger than, equal to, or less than $\press^\mathrm{set}$, the mode is closed, active, or open, respectively.
In addition, these regulators also feature a check valve behavior, which in their case allows flow against the regulator's orientation (bypass instead of open mode).
Except for this back flow, the described behavior can as well be achieved with our formulation, for example, by setting $\tvI{\MAXPOUT}=\press^\mathrm{set}$, $\tvI{\MINPIN} = 0$, and $\tvI{\MINPOUT} = \infty$.
The target value \MINPIN is always satisfied and \MINPOUT is always violated.
Hence, we are either in active mode and have a stable \MAXPOUT, with $\pressI{r}$ equal to the set-point value, or \MAXPOUT is violated, leading to a closed regulator, or $\pressI{r}$ is below \MAXPOUT and satisfied, leading to the open mode.
This shows that our optimization model covers regulators with and without remote access.

\section{Numerical Evaluation}\label{sec:numeric_evaluation}

To show and verify the accuracy of our models, we present in this section two experiments using the target value control on a single regulator.
It is embedded in a simple path network between two pipes, each with $L=\SInumber{10}{km}, D=\SInumber{0.9}{m}, r=\SInumber{0.012}{mm}$, see Figure~\ref{fig:numeric_example_net}.
\begin{figure}[h!]
    \centering
    \begin{tikzpicture}[circuit, circuit symbol unit=6pt]
        \coordinate (LL)  at (0, 0);
        \coordinate (L)   at (3, 0);
        \coordinate (R)   at (6, 0);
        \coordinate (RR)  at (9, 0);

        \draw (LL) to [pi] (L);
        \draw (L) to [cv] (R);
        \draw (R) to [pi] (RR);
        
        \node[fill=white, draw, circle, label=above:{$n_\textrm{in}$}] at (LL)  {};
        \node[fill=white, draw, circle, label=above:{$n_\ell$}] at (L)   {};
        \node[fill=white, draw, circle, label=above:{$n_r$}] at (R)   {};
        \node[fill=white, draw, circle, label=above:{$n_\textrm{out}$}] at (RR)  {};
    \end{tikzpicture}
    \caption[Benchmark network]{Benchmark network, consisting of one regulator arc \inlinegraphics{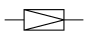} and two pipes \inlinegraphics{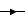} surrounding the regulator.}
    \label{fig:numeric_example_net}
\end{figure}

For both scenarios, we determine the solution using the dynamic simulation model presented in Section~\ref{sec:simulationModel} as well as the optimization model presented in Section~\ref{sec:optimizationModel}.
The reference solution is determined using the SIMONE simulator\footnote{LIWACOM Informationstechnik GmbH, SIMONE software, \url{https://www.liwacom.de/}, accessed: 2021-11-22}, which has established itself as a de facto standard due to its wide use in industry. 
We will refer to solutions computed with SIMONE as industry simulations as a convenient shorthand.

The initial steady state of the scenarios features a flow of $\SInumber{10}{kg/s}$ on all arcs as well as inflow at the entry and outflow at the exit.
The pressure ranges from $\SInumber{50}{bar}$ at $n_\mathrm{in}$ to $\SInumber{49.992}{bar}$ at $n_\mathrm{out}$, the regulator is fully open, the constant gas temperature is set to $\SInumber{283.15}{K}$, and the gas mixture is, for the sake of simplicity, assumed to be pure methane.

As pipe model for the simulation, we use the left-right-discretization \eqref{eqn:pipe_ode_LR}, equipped with the linear AGA formula for the real gas factor.
Note that we chose $\alpha=1.0\mathrm{e}3$ for the simulation model.

For the optimization, we take the linearized pipe model \eqref{eqn:new_momentum_const_v} discretized with the implicit box scheme.
Furthermore, we use a constant real gas factor $z$, which is determined for each node using the formula of Papay and computed from the initial state.
The linearized model is not reliable in the case of volatile flow conditions but works fine for the following scenarios featuring similar flows for most of the time.
As a benefit, the resulting overall model is linear, which enables us to use a MILP solver for the optimization, in our case Gurobi \citep{Gur2020}. 
Furthermore, we picked $\varepsilon=0$ for the optimization model.

\subsection{First Scenario}
For the first scenario, we use a set of target values obtained from our project partner, the gas network operator Open Grid Europe GmbH (OGE).
These target values should lead to different opening degree and throughput changes of the regulator while assuming a steady inflow and outflow at the boundaries of $\SInumber{10}{kg/s}$ over the whole time horizon of $\SInumber{12}{h}$. 
This scenario is a stress test.
It has been designed to provoke relevant interferences from the regulator involving almost each available target value. 
An overview of the different switching actions of the target values is provided in Table~\ref{tab:numeric_example_OGE_TVs}.
\begin{table}[ht]
    \centering
    \newcommand{\myBar}{\mathmakebox[0.5cm][l]{\,\SIunit{bar}}}
    \newcommand{\myKgS}{\mathmakebox[0.5cm][l]{\,\SIunit{kg/s}}}
    \begin{tabular}{ccr}
        TV & Time as hh:mm & Value \\
        \MINPIN  & 00:00 &  $48.0\myBar$\\
        \MAXPOUT & 00:00 &  $55.0\myBar$\\
        \MAXPIN  & 00:00 & $100.0\myBar$\\
        \MINPOUT & 00:00 &  $40.0\myBar$\\
        \MAXQ    & 00:00 &   $9.0\myKgS$\\
        \MAXQ    & 01:00 &  $15.0\myKgS$\\
        \MAXQ    & 02:00 &   $6.0\myKgS$\\
        \MAXQ    & 02:30 &  $10.0\myKgS$\\
        \MAXPOUT & 03:30 &  $47.0\myBar$\\
        \MAXPOUT & 04:30 &  $55.0\myBar$\\
        \MINPIN  & 05:00 &  $55.0\myBar$\\
        \MINPIN  & 05:30 &  $53.0\myBar$\\
        \MINPOUT & 06:30 &  $46.0\myBar$\\
        \MAXQ    & 06:30 &   $6.0\myKgS$\\
        \MINPOUT & 07:00 &  $46.5\myBar$\\
        \MINPOUT & 07:30 &  $47.5\myBar$\\
    \end{tabular}
    \caption{The target values specified by Open Grid Europe GmbH (OGE) in our benchmark scenario. Each row represents a target value change of the given type of target value to the given value at the given time. The time is specified as hours and minutes past the initial state.}
    \label{tab:numeric_example_OGE_TVs}
\end{table}
For the optimization, the simulation, as well as the reference industry simulation, we use a time discretization of $\SInumber{3}{min}$. Note that the target values are fixed for the optimization here, which effectively transforms the problem into a plain feasibility problem.

%
%
\begin{figure}[!htb]
    \centering
    \includegraphics[trim={0.0cm 0.0cm 0.0cm 0.4cm},clip,width=\linewidth]{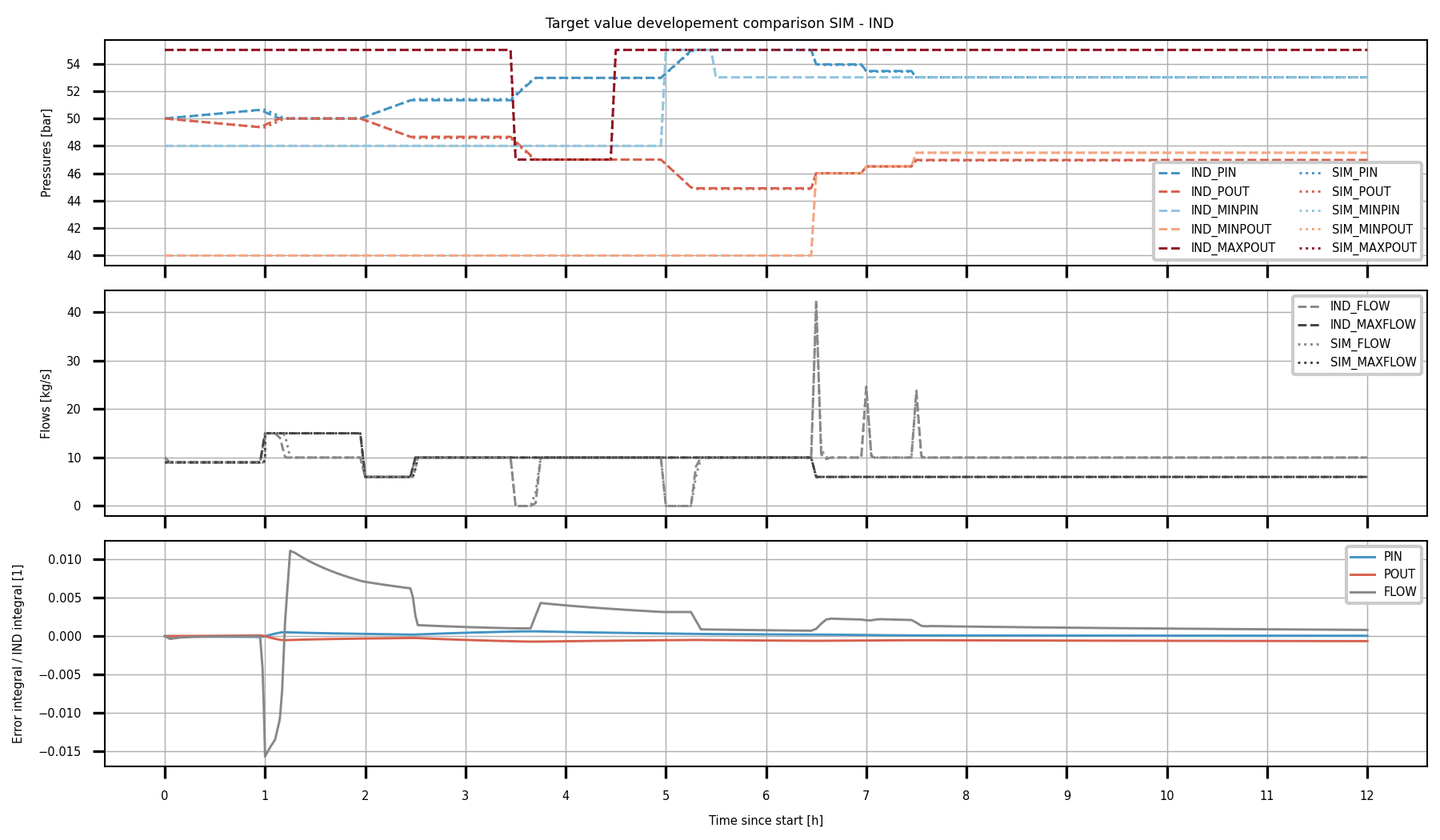}
    \caption{Comparison of the simulated solution (SIM) with the reference simulation from the industry simulator (IND) for the first scenario with predefined target values based on Table~\ref{tab:numeric_example_OGE_TVs}. The pictures show the left and right pressure values and target values of the regulator, the flow values and target values, and the difference of both solutions in the three quantities integrated over time and normalized by the quantity integral of the IND values.}
    \label{fig:OGE_Scen_SIMvsIND}
\end{figure}

\begin{figure}[!htb]
    \centering
    \includegraphics[trim={0.0cm 0.0cm 0.0cm 0.4cm},clip,width=\linewidth]{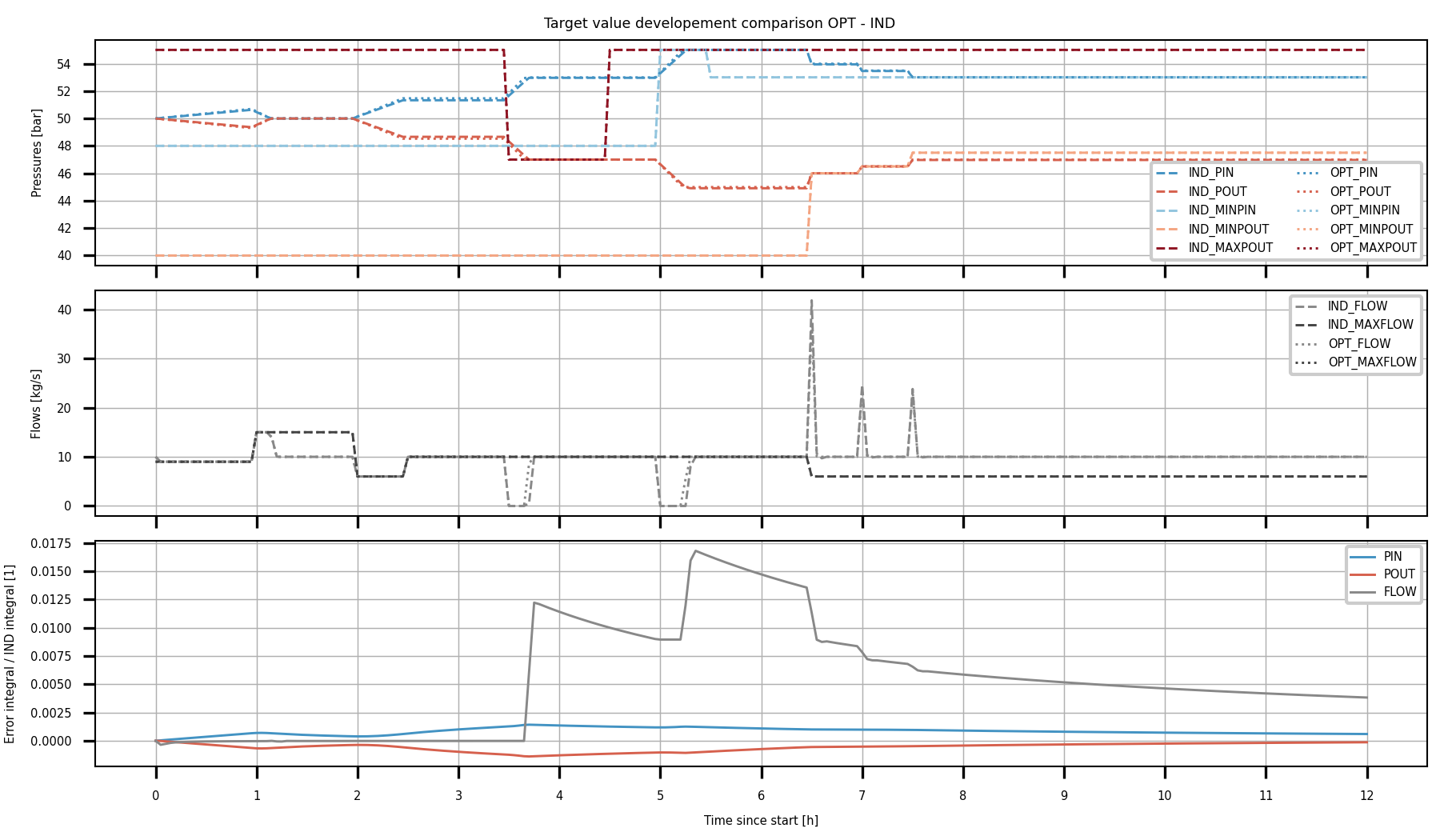}
    \caption{Comparison of the optimized solution (OPT) with the reference simulation from the industry simulator (IND) for the first scenario with predefined target values based on Table~\ref{tab:numeric_example_OGE_TVs}. The pictures show the left and right pressure values and target values of the regulator, the flow values and target values, and the difference of both solutions in the three quantities integrated over time and normalized by the quantity integral of the IND values.}
    \label{fig:OGE_Scen_OPTvsIND}
\end{figure}

\begin{figure}[!htb]
    \centering
    \includegraphics[trim={0.0cm 0.0cm 0.0cm 0.4cm},clip,width=\linewidth]{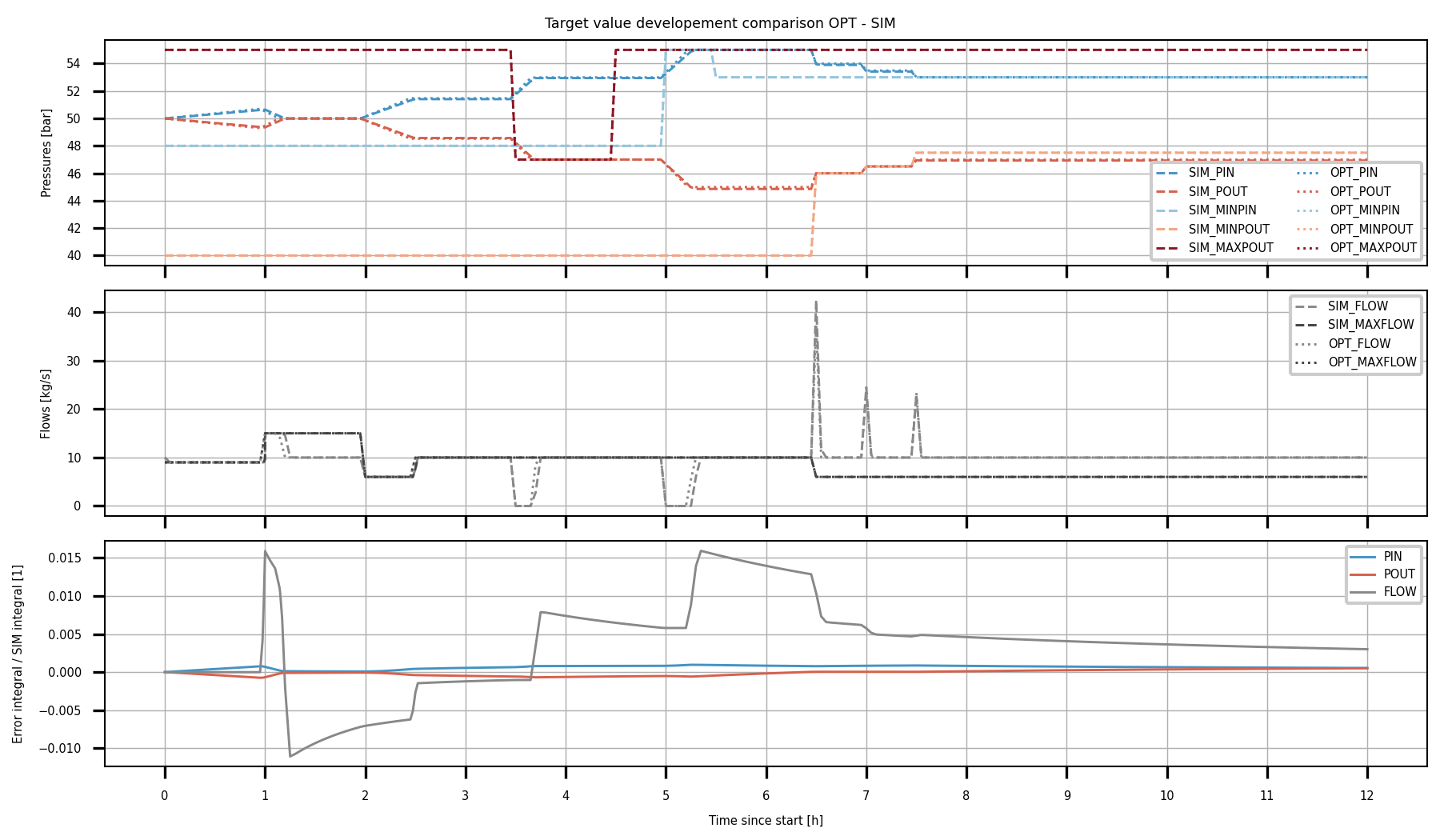}
    \caption{Comparison of the optimized solution (OPT) with the reference simulated solution (SIM) for the first scenario with predefined target values based on Table~\ref{tab:numeric_example_OGE_TVs}. The pictures show the left and right pressure values and target values of the regulator, the flow values and target values, and the difference of both solutions in the three quantities integrated over time and normalized by the quantity integral of the SIM values.}
    \label{fig:OGE_Scen_OPTvsSIM}
\end{figure}

The results for the first scenario are shown in Figure~\ref{fig:OGE_Scen_SIMvsIND} comparing the simulated solution (SIM) with the one computed by the industry simulator (IND), Figure~\ref{fig:OGE_Scen_OPTvsIND} compares the optimized solution (OPT) with the one from the industry simulator, and Figure~\ref{fig:OGE_Scen_OPTvsSIM} compares the optimized and simulated solution.
Each figure shows the development of the left and right pressure values and target values of the regulator in the upper box and the development of the flow values and target values in the middle box.
The lowest box displays the relative error defined as the integrated difference over time divided by the time-integrated reference solution, i.e., let \(\mathrm{SOL} \neq \mathrm{REF} \in \{\mathrm{SIM}, \mathrm{OPT}, \mathrm{IND}\} \) and \(x \in \{p_\ell, p_r, q\}\), then
\begin{align*}
\text{relative error}[x](t) \equiv \frac{\int_{t_0}^t \mathrm{SOL}[x](s) - \mathrm{REF}[x](s) \mathrm ds}{\int_{t_0}^t \mathrm{REF}[x](s) \mathrm ds}.     
\end{align*}
At the initial start, all the pressure target values are fulfilled, but $\MAXQ=\SInumber{9}{kg/s}$ is violated.
Hence, the regulator closed until $q=\MAXQ$.
At 01:00, \MAXQ is set to $\SInumber{15}{kg/s}$.
Hence, the regulator fully opens again, the pressures equalize, and we go back to the initial steady state.
Next, \MAXQ is set to $\SInumber{6}{kg/s}$ for 30 minutes and then goes to $\SInumber{10}{kg/s}$.
The regulator enforces corresponding values of $q$, which first create a pressure imbalance by reducing the flow and then stabilizes this imbalance.
At 3:30, we start using the pressure target values by setting \MAXPOUT to $\SInumber{47.0}{bar}$, which is violated.
Since \MAXPOUT has the highest priority, the regulator closes until $p_r=\SInumber{47.0}{bar}$ and then keeps this value, i.e., $q=\SInumber{10}{kg/s}$ is established again.
After setting back \MAXPOUT, the situation does not change since \MAXQ still enforces $q=\SInumber{10}{kg/s}$.
Now, $\MINPIN=\SInumber{55.0}{bar}$ is set, which again is violated and of high priority.
Therefore, we further increase the pressure imbalance until reaching $p_\ell=\SInumber{55.0}{bar}$ and keep this situation stable, even though \MINPIN is reduced to $\SInumber{53}{bar}$ again shortly after.
The next new setting is \MINPOUT to $\SInumber{46.0}{bar}$ at 06:30, which is a violation of the target value.
Since \MINPIN and \MAXPOUT are currently fulfilled, this forces the regulator to open until $p_r=\SInumber{46.0}{bar}$ is reached.
Here, we stabilize the flow at $q=\SInumber{10}{kg/s}$ again.
The reduction of \MAXQ to $\SInumber{6}{kg/s}$ does not change this since \MINPOUT is of higher priority than \MAXQ.
As final changes, we first set \MINPOUT to $\SInumber{46.5}{bar}$ and then to $\SInumber{47.5}{bar}$.
The first change shifts the stable situation to $p_r=\SInumber{46.5}{bar}$.
However, we do not reach $p_r=\SInumber{47.5}{bar}$ afterwards since this would violate $\MINPIN=\SInumber{53}{bar}$.
Hence, we stabilize at $p_\ell=\SInumber{53.0}{bar}$ and keep this situation until the end of the time horizon.

All three solutions are very similar to each other in terms of the development of the regulator's operation point over time.
Hence both derived models, which are the simulation model and the optimization model, do catch the regulator's behavior compared to the industrial simulation.
As an overview, we listed below the maximal relative error and the relative error at the end of the time horizon for the pressure values, where we always choose the bigger of the two errors given for the left and right pressure, as well as the flow values.
\begin{center}
\begin{tabular}{rrrrr}
            & max $p$ error & max $q$ error & end $p$ error & end $q$ error \\
 SIM vs IND & 0.07\,\% & 1.57\,\% & 0.06\,\% & 0.08\,\% \\
 OPT vs IND & 0.14\,\% & 1.68\,\% & 0.06\,\% & 0.38\,\% \\
 OPT vs SIM & 0.10\,\% & 1.59\,\% & 0.05\,\% & 0.30\,\%
\end{tabular}
\end{center}
The values show that the differences in the pressures are, in general, smaller than those in the flows.
This is due to the fact that each of the three methods has a different internal implementation of the reaction of the regulator's opening degree to changing target values conditions.
Hence, even though the time discretization is the same for all approaches, the flow values are different in the unstable periods shortly after a target value change leading to a new violation.
However, these differences only occur for brief periods of time.
We also note that the differences between the two simulation procedures are, in general, smaller than those between the optimization and one of the simulation runs.
This is no surprise, as we assume an idealized target value behavior for the optimization model, see Assumption~\ref{ass:perfectControl}.
However, we are glad to see how close the solution of the idealized model is to both simulations.

\subsection{Second Scenario}
In the second scenario, we tested both models in a different setup.
Here, we determine the regulator's target values using the optimization model and verify them using the simulation model.
Since this task is quite demanding in a volatile environment, even for only a single regulator, we discretized the time horizon for the optimization model into steps of 15 minutes, leading to 48 time steps in total, whereas the simulation still solves in $\SInumber{3}{min}$ time step resolution.

To ensure and provoke the necessity of target value adjustments, we prescribed not only the inflow at the entry and outflow at the exit but also the pressure values over time at both nodes.
We took these pressure values of the OPT solution from scenario 1 with time step resolution of \(\SInumber{15}{min}\).
This enables us to compare the number of target value changes between the OPT solutions of scenarios 1 and 2 and provide an estimate of the optimization potential.

There is a fundamental difference in the interpretation of value updates of target values between the optimization and the simulation model.
A simulation starts with the control adjustment at some point in time, let it be $t$, and proceeds over the following time steps with the adjustment to the new target value(s).
In contrast to this and due to Assumption~\ref{ass:perfectControl}, the optimization model assumes a control which has already adjusted to the new target value directly at time $t$.
Hence, it basically assumes that the target values have been changed at some point within the 15 minute time step preceding time $t$.
Here, we have a \emph{degree of freedom} and could eliminate this by a subsequent \emph{event detection} approach to narrow down the event-interval of the switch to fit the time-resolution of the following simulation.
Instead, we use a heuristic and will communicate changes in the target values that occurred at time $t$ in the optimization model to both simulation processes with time $t - \Delta_t/2$, where $\Delta_t$ is the time resolution of the optimization model, i.e., $\Delta_t=\SInumber{15}{min}$.
This minimizes the maximum possible error between the used and the actual change time.
Note that we round $\SInumber{7.5}{min}$ to the integer value of $\SInumber{8}{min}$ here.

\begin{table}[ht]
    \centering
    \newcommand{\myBar}{\mathmakebox[0.5cm][l]{\,\SIunit{bar}}}
    \newcommand{\myKgS}{\mathmakebox[0.5cm][l]{\,\SIunit{kg/s}}}
    \begin{tabular}{cccr}
        TV & \multicolumn{2}{c}{Time as hh:mm} & Value \\
           & OPT               & SIM \\
        \MINPIN  & 00:00 & 00:00 &  $48.0\myBar$\\
        \MAXPOUT & 00:00 & 00:00 &  $55.0\myBar$\\
        \MAXPIN  & 00:00 & 00:00 & $100.0\myBar$\\
        \MINPOUT & 00:00 & 00:00 &  $40.0\myBar$\\
        \MAXQ    & 00:00 & 00:00 &   $9.0\myKgS$\\
        \MAXQ    & 01:00 & 00:52 &  $13.0\myKgS$\\
        \MAXQ    & 02:00 & 01:52 &   $6.0\myKgS$\\
        \MINPOUT & 02:30 & 02:22 &  $48.5\myBar$\\
        \MAXPOUT & 03:30 & 03:22 &  $47.0\myBar$\\
        \MAXPOUT & 05:00 & 04:52 &  $45.0\myBar$\\
        \MAXPOUT & 06:30 & 06:22 &  $46.0\myBar$\\
        \MAXPOUT & 07:00 & 06:52 &  $46.5\myBar$\\
        \MAXPOUT & 07:30 & 07:22 &  $47.0\myBar$
    \end{tabular}
    \caption{The target values determined by the optimization model. Each row represents a target value change of the given type of target value to the given value at the given time. The time is specified as hours and minutes past the initial state. The OPT time is the one determined by the optimization, and the SIM time is the one communicated to both simulation processes.}
    \label{tab:numeric_example_Opt_TVs}
\end{table}
%
%
\begin{figure}[!htb]
    \centering
    \includegraphics[trim={0.0cm 0.0cm 0.0cm 0.4cm},clip,width=\linewidth]{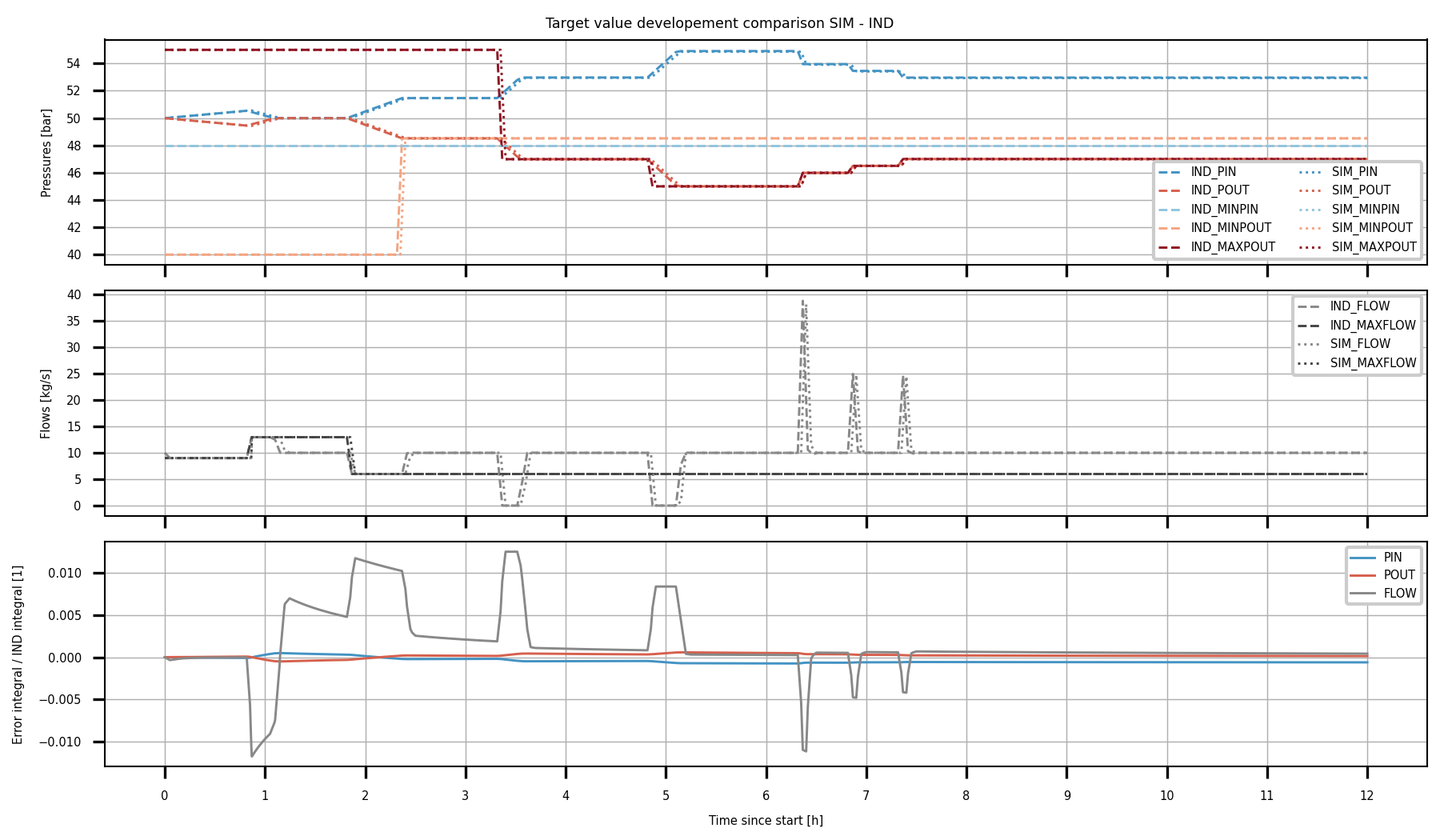}
    \caption{Comparison of the simulation (SIM) with the reference simulation from the industry simulator (IND) for the second scenario with optimized target values, see Table~\ref{tab:numeric_example_Opt_TVs}. The pictures show the left and right pressure values and target values of the regulator, the flow values and target values, and the difference of both solutions in the three quantities integrated over time and normalized by the quantity integral of the IND values.}
    \label{fig:Opt_Scen_SIMvsIND}
\end{figure}

\begin{figure}[!htb]
    \centering
    \includegraphics[trim={0.0cm 0.0cm 0.0cm 0.4cm},clip,width=\linewidth]{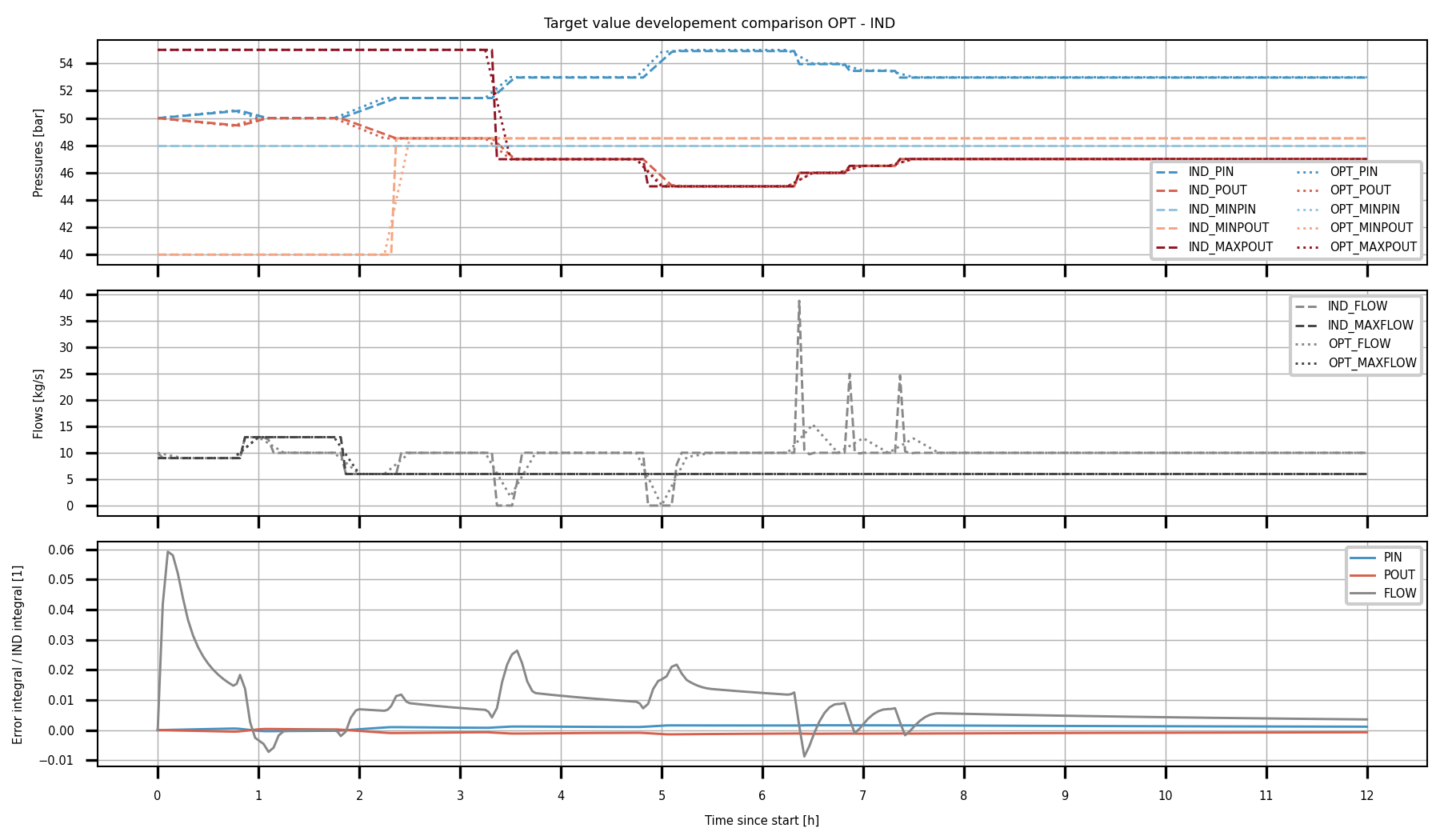}
    \caption{Comparison of the optimized solution (OPT) with the reference simulation from the industry simulator (IND) for the second scenario with optimized target values, see Table~\ref{tab:numeric_example_Opt_TVs}. The pictures show the left and right pressure values and target values of the regulator, the flow values and target values, and the difference of both solutions in the three quantities integrated over time and normalized by the quantity integral of the IND values.}
    \label{fig:Opt_Scen_OPTvsIND}
\end{figure}

\begin{figure}[!htb]
    \centering
    \includegraphics[trim={0.0cm 0.0cm 0.0cm 0.4cm},clip,width=\linewidth]{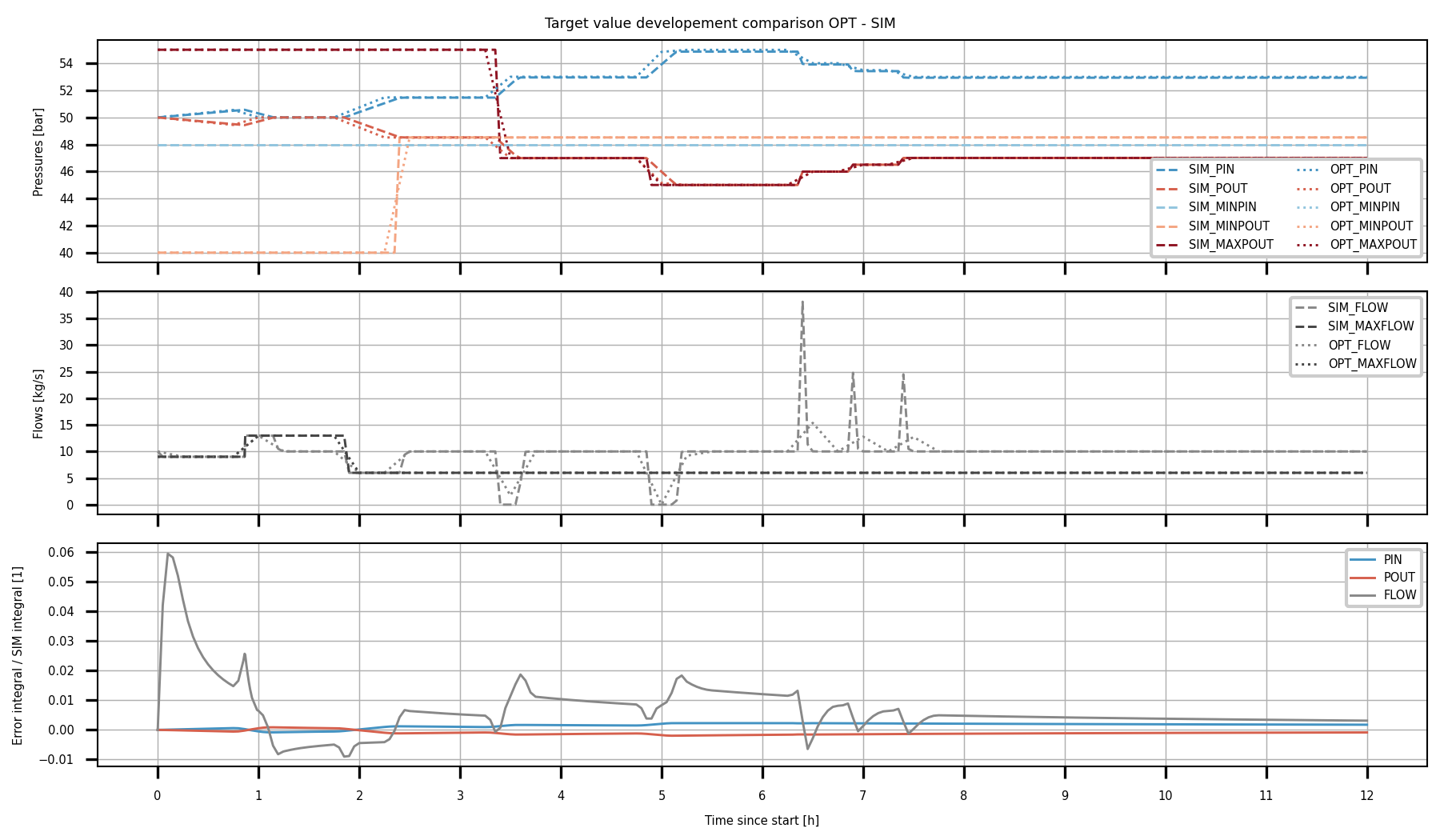}
    \caption{Comparison of the optimized solution (OPT) with the reference simulated solution (SIM) for the second scenario with optimized target values, see Table~\ref{tab:numeric_example_Opt_TVs}. The pictures show the left and right pressure values and target values of the regulator, the flow values and target values, and the difference of both solutions in the three quantities integrated over time and normalized by the quantity integral of the SIM values.}
    \label{fig:Opt_Scen_OPTvsSIM}
\end{figure}
The target values determined by the optimization are listed in Table~\ref{tab:numeric_example_Opt_TVs}.
The comparisons of the different solutions over time are shown in Figure~\ref{fig:Opt_Scen_SIMvsIND} for the simulation (SIM) and the industry simulation (IND), in Figure~\ref{fig:Opt_Scen_OPTvsIND} for optimized solution (OPT) and the industry simulation, and in  Figure~\ref{fig:Opt_Scen_OPTvsSIM} for the optimized and simulated solution.

When comparing the listed target value changes in Tables~\ref{tab:numeric_example_OGE_TVs} and~\ref{tab:numeric_example_Opt_TVs}, we observe that we only needed 8 target value changes instead of the previous 11 to archive the regulator control.
In the optimized solution, we do not use the changes of \MAXPOUT at 04:30 and \MAXQ at 06:30 as they do not affect the regulator's control. 
In addition, the control change at 07:30, which was a consequence of changing \MINPIN at 05:30 and \MINPOUT at 07:30, could be realized by a single change of \MAXPOUT at 07:30.
Similar to this last case, some of the other control changes are now triggered by different target value changes.
Another example would be the change at 05:00, which was previously caused by setting \MINPIN to $\SInumber{55}{bar}$ and is now triggered by setting \MAXPOUT to $\SInumber{45}{bar}$.

When looking at the solutions over time, the development is again quite similar, which can also be observed from the table below.
It shows the maximal relative error and the relative error at the end of the time horizon for the pressure values, where we always choose the bigger of the two errors given for the left and right pressure, as well as the flow values.
\begin{center}
\begin{tabular}{rrrrr}
            & max $p$ error & max $q$ error & end $p$ error & end $q$ error \\
SIM vs IND & 0.07\,\% & 1.26\,\% & 0.06\,\% & 0.04\,\% \\
OPT vs IND & 0.16\,\% & 5.92\,\% & 0.11\,\% & 0.35\,\% \\
OPT vs SIM & 0.23\,\% & 5.95\,\% & 0.17\,\% & 0.31\,\%
\end{tabular}
\end{center}
The error values of the comparison between our simulation and the one of the industry software are consistent with those of the first scenario.
This is not surprising, as only the simulated instance changes while the general setup for both stays the same.
When comparing one of the two and the optimized solution, we clearly see the difference to the values of the first scenario.
Especially the maximum flow error, but also the maximal and overall errors in the pressure have increased significantly.
This is an expected change, as the optimization now uses a more coarse discretization.
Nevertheless, the pressure differences are still very small in general and not too far away from the values of either simulation.
Regarding the errors in the flow, the main difference occurs in the very first minutes of the scenario.
Since the initial target values are violated, the control adjusts immediately, which happens during the first time step for all of the three approaches.
Due to the different time discretizations, the adjustment is 5 times faster for the simulations than for the optimization.
However, the overall error regarding the flow values is surprisingly small and consistent to the values of the first scenario.
We summarize that both simulations yield very similar results, while the solution of the target value optimization may produce different values during the times of adjusting the control to new target values, but gives very similar results regarding the stabilized control situations.

\section{Outlook}\label{sec:outlook}
In this article, we introduced models for regulators in gas networks and embedded target values for their control.
The target values form a system of 6 simultaneously active and hence competing set-points, which enables the regulating element to react to changing conditions in the surrounding network without the need of continuous human observation or manual interactions.
Our description lives up to expectations and needs compared to features found in industry-standard simulators.

To capture the behavior of target value controlled regulators, we introduced two different models.
The first one is a model suitable for dynamic or transient simulation.
Here, the target values are assumed to be piecewise linear functions, and their competing logic is encoded in terms of nested \(\max-\min\)-comparisons.
A second model was designed to be used in discrete optimization problems and determines an optimized set of target values changing as little as possible while enabling the demanded network control.
It is based on a characterization of the target value behavior as a set of specific element states, the stable-pushing target value combinations.
These assume a regulator reacting to changes in its environment or of its target values immediately and perfectly precise.
Depending on the chosen combination, different constraints are active that relate the target values to the values of the element's point of operation.

Both models have been evaluated in comparison to a commercial simulator, actively used in industry. 
We showed that the behavior of our simulation model matches the expectations in that the simulated solutions stick close in terms of the resulting pressure and flow values.
The relatively higher differences in the flow values are short-lived and their integral equaling the total amount of transported gas is negligible.
These differences in the flow values represent slightly different timings and adjustment speeds of regulator actions.
Possible strategies to address this could be the introduction of delay or slope limiter.
However, this is more a question of technical specification rather than mathematical consideration.
For the optimization model, we could show that it works very well when using the same time discretization as for the simulation together with already fixed target values.
Here, the differences to both simulated results are just slightly larger than both simulations among each other, even though we used a model based on more simplifying assumptions.
In a more realistic scenario, in which the target values are subject to optimization, we need to use fewer and larger time steps.
As a result, the relative differences to the simulated results have increased.
However, they still stay very close and bounded, especially the components representing the pressures. 
This indicates that the optimization model is able to reproduce the target value behavior reasonably well.

There are multiple paths to pursue for future research.
Firstly, we could replace or add other kinds of target values, e.g., bounding values for the ratio or difference of in- and outlet pressures.
Secondly, the regulator model for simulation could be converted into a model for band regulator if \(\MINQ \le \MAXQ\) by replacing \eqref{eqn:cv_derivation_stage_i} with
\begin{align*}
    \dot q = \alpha[\min(0, \MAXQ - q) + \max(0, \MINQ - q)],
\end{align*}
just before adding the remaining target value mechanisms.
Thirdly, we may apply our techniques to embed target-value-based controls into compressor unit, group, and station modeling. 
In that regard, we have already discussed it briefly in Section~\ref{sec:compressors}, but it would need to be elaborated and established further in full detail.
Furthermore, the optimization model can probably be improved to enhance the performance when combined with state-of-the-art optimization frameworks.
While the model captures the target value behavior very well already, it is quite challenging to solve for one single element already.
The overall goal is to make it usable in a transient gas network operation model using multiple regulators and hundreds of other network elements.
Finally, the concept of a target-value-based element control can be applied to active elements in other fluid-based flow networks.
As an example, we mention pumps in water networks, which from a macroscopic viewpoint similarly control their network as compressors do for gas networks.

\section*{Acknowledgements}
The work for this article has been conducted in the Research Campus MODAL funded by the German Federal Ministry of Education and Research (BMBF) (fund numbers 05M14ZAM \& 05M20ZBM).
Furthermore the authors thank the Deutsche Forschungsgemeinschaft for their support within Project B10 and Project C02 in the Sonderforschungsbereich / Transregio 154 Mathematical Modelling, Simulation and Optimization using the Example of Gas Networks.

\typeout{}
\bibliography{bibliography}

\end{document}